\documentclass{amsart}

\usepackage[utf8]{inputenc}  
\usepackage[T1]{fontenc} 
\usepackage{amsmath,amssymb,amsthm} 
\usepackage[pagebackref=false]{hyperref} 
\usepackage{microtype} 
\usepackage[inline]{enumitem} 
    \setlist[enumerate]{label=\textup{(\roman*)}}
\usepackage[all,cmtip]{xy} 
\usepackage{stmaryrd}

\numberwithin{equation}{section}

\renewcommand{\geq}{\geqslant}
\renewcommand{\leq}{\leqslant}
\renewcommand{\ge}{\geqslant}
\renewcommand{\le}{\leqslant}

\newcommand*{\nn}{\mathbb{N}}

\newcommand*{\from}{\colon} 
\newcommand*{\green}[1]{\ensuremath{\mathrel{\mathcal{#1}}}}
\newcommand*{\pro}[1]{\widehat{#1}}

\newcommand*{\pv}[1]{\ensuremath{\mathsf{#1}}} \newcommand*{\loc}{{L}}
\newcommand*{\init}{\alpha} \newcommand*{\term}{\omega}
\newcommand*{\edg}{E} 
\newcommand*{\rstab}{\operatorname{Stab}}
\newcommand*{\om}[2]{\ensuremath{\Omega_{#1}{\pv{#2}}}}
\newcommand*{\Om}[2]{\ensuremath{\overline{\Omega}_{#1}{\pv{#2}}}}
\newcommand*{\ltr}[1]{\mathtt{#1}} 
 
\newcommand*{\dex}[1]{\ensuremath{\mathsf{Pa}}_{\pv{#1}}}
\newcommand*{\Cl}[1]{\ensuremath{\mathcal{#1}}}

\newcommand*{\clos}[1]{\operatorname{Cl}_{\pv{#1}}}
\newcommand*{\glob}{\mathit{g}}

\newcommand\malcev{%
\mathbin{\hbox{$\bigcirc$\kern-9pt\raise1.2pt\hbox{\scriptsize$m$}\,}}}

\theoremstyle{plain}
\newtheorem{theorem}{Theorem}[section]
\newtheorem{lemma}[theorem]{Lemma}
\newtheorem{proposition}[theorem]{Proposition}
\newtheorem{corollary}[theorem]{Corollary}
\newtheorem{problem}[theorem]{Problem}

\theoremstyle{definition}
\newtheorem{definition}[theorem]{Definition}

\theoremstyle{remark} 
\newtheorem{remark}[theorem]{Remark}
\newtheorem{example}[theorem]{Example}

\begin{document}

\title{Open multiplication in relatively free profinite semigroupoids}

\author[J. Almeida]{Jorge Almeida}
\address{CMUP, Dep.\ Matem\'atica, Faculdade de Ci\^encias, Universidade do Porto, Rua do Campo Alegre 687, 4169-007 Porto, Portugal}
\email{jalmeida@fc.up.pt}

\author[A. Costa]{Alfredo Costa}
\address{University of Coimbra, CMUC, Department of Mathematics, Largo D. Dinis, 3000-143, Coimbra, Portugal}
\email{amgc@mat.uc.pt}

\author[H. Goulet-Ouellet]{Herman Goulet-Ouellet}
\address{Universit\'e de Moncton, 60 Notre-Dame-du-Sacr\'e-Coeur Street, Moncton E1A 3E9, New-Brunswick, Canada}
\email{herman.goulet-ouellet@umoncton.ca}

\begin{abstract}
  The purpose of this paper is to extend some useful results, such as
  the multiplication being open, previously known for suitable
  finitely generated relatively free profinite semigroups, to
  relatively free profinite semigroupoids over finite-vertex graphs.
  This extension is used to give a profinite characterization of
  recurrent words over infinite alphabets and to establish new results
  about stabilizers in relatively free profinite semigroups and
  semigroupoids.
\end{abstract}

\maketitle

\section{Introduction}
\label{sec:intro}

Pseudovarieties of semigroups are considered one of the most fruitful
frameworks of finite semigroup theory (see the recent survey~\cite{Almeida:2025}). 
Several results have shown the convenience and advantages to enlarge our scope
to semigroupoids, profinite semigroups, and even profinite
semigroupoids. From the point of view of Category Theory, semigroupoids are simply the result of dropping from the
definition of small categories the requirement of existence of local
identities. Their exploration as partial algebras generalizing
semigroups was initiated in the decade of 1980 by
Tilson~\cite{Tilson:1987}, with powerful applications in the study of finite
semigroups. In the same decade, building on work of
Reiterman and Banaschewski \cite{Reiterman:1982,Banaschewski:1983},
the first author promoted the development of relatively free profinite
semigroups as a means of giving a proper equational, syntactical,
approach to pseudovarieties of semigroups. Both approaches were
afterwards successfully employed in several major developments of the
field~\cite{Rhodes&Steinberg:2009qt,Almeida:1996c,Almeida&ACosta:2015hb}.
The papers~\cite{Almeida&Weil:1996,Jones:1996} pioneered the combination
of the two approaches, establishing the study of free profinite
semigroupoids relatively to pseudovarieties of semigroupoids.
Other relevant examples of this combination include the papers~\cite{Rhodes&Steinberg:2001,Almeida&Steinberg:2000a}.

Motivation for those studies stems from Eilenberg's theorem relating
varieties of languages with pseudovarieties of
semigroups~\cite{Eilenberg:1976}. In such varieties, only languages
over finite alphabets are considered. This explains the focus given in
the literature to relatively free profinite semigroups generated by
\emph{finite} alphabets. In this context, close connections are
established between the algebraic-topological structure of free
profinite semigroups over a pseudovariety $\pv V$ and the
corresponding variety of languages $\Cl V$. An example, established
in~\cite{Almeida&ACosta:2007a}, is that when $\pv V$ contains all
finite nilpotent semigroups, $\Cl V$ is closed under concatenation of
languages if and only if the corresponding relatively free profinite
semigroups have open multiplication (i.e., the product of open sets is
open). 
This result proved useful to understand the structure of
those semigroups by allowing us to think of their
elements, often aptly called \emph{pseudowords}, as analogous to finite words in some key ways.

Relatively free profinite semigroups over infinite alphabets were also
studied, but they present some significant challenges; they may not be metrizable,
while the finitely generated ones always are (see, for
instance~\cite{Almeida&Steinberg:2008}). Here, we present extensions
to (non-necessarily finitely generated) relatively free profinite
semigroups, over suitable pseudovarieties, of results previously
established in the finitely generated case only. In fact, we go
further, making these generalizations for relatively free profinite
\emph{semigroupoids} generated by finite-vertex (directed) \emph{graphs}. Graphs
generalize alphabets, as we see alphabets as one-vertex graphs with loops representing the letters.
Accordingly, the generalization of pseudowords to relatively free
profinite semigroupoids are called \emph{pseudopaths}. Most of the
aforementioned extensions concern the link between open multiplication
(in semigroups and semigroupoids) and concatenation-closed
pseudovarieties (of semigroups and of semigroupoids, respectively). We
get complete characterizations for such pseudovarieties
(cf.~Theorems~\ref{t:concatenation-closed-pseudovarieties}
and~\ref{t:sgpd-concatenation-closed-pseudovarieties}), improving
what, in the semigroup case, was limited to finite alphabets only.

Our characterization of open multiplication in semigroupoids via
\emph{fully factorizable} nets
(Corollary~\ref{c:open-mult-fully-factorizable}) improves the one
from~\cite[Lemma 3.2]{Almeida&ACosta&Costa&Zeitoun:2019}, which is
limited to metrizable semigroups. This new characterization is used in
the proofs of several results concerning what we call \emph{prefix accessible}
pseudopaths (i.e., infinite-length cluster points of prefixes of right
infinite paths). 
One such result is a profinite characterization of recurrent right-infinite paths over finite-vertex graphs (Theorem~\ref{t:idempotents-in-Pw-pseudopath-version}). 
This is motivated by parallel ongoing work on a profinite approach to the study of S-adic sequences~\cite{Almeida&ACosta&Goulet-Ouellet:2024b}, where we deal with relatively free profinite semigroupoids over finite-vertex graphs having possibly infinitely many edges (which are not necessarily metrizable).

The culmination of the paper is
Theorem~\ref{t:net-characterization-of-L-minimal-right-stabilizers}, which
gives a characterization (also applied
in~\cite{Almeida&ACosta&Goulet-Ouellet:2024b}) of stabilizers of
prefix accessible pseudopaths for concatenation-closed
pseudovarieties. Fully factorizable nets appear in its proof, in whose preparation
we also get a structural result about stabilizers of
free profinite semigroups over equidisivible pseudovarieties
(Theorem~\ref{t:right-stabilizers-L-chain} and
Corollary~\ref{c:category-right-stabilizers-L-chain}). This overlaps
with results obtained by Rhodes and Steinberg using other methods,
cf.~\cite{Rhodes&Steinberg:2001} and
Remark~\ref{r:stabilizers-rhodes-steinberg}. Stabilizers of relatively
free profinite semigroups got significant attention in the
literature~\cite{Rhodes&Steinberg:2001,Henckell&Rhodes&Steinberg:2010b};
they played a role, via \emph{stable pairs}, in the recently announced
proof that the Krohn-Rhodes complexity is
decidable~\cite{Margolis&Rhodes&Schilling:arXiv:2406.18477,Margolis&Rhodes&Schilling:arXiv:2501.00770v1}.

The paper is organized as follows. After this introduction,
Section~\ref{sec:prelims} sets some basic definitions and notation.
This is continued in Section~\ref{sec:relat-free-prof}, dedicated to
relatively free profinite semigroupoids.
Section~\ref{sec:open-multiplication}, the core of the paper, contains
a systematic study of links between open multiplication and fully
factorizable nets on one hand, and concatenation-closed pseudovarieties on the other hand.
Section~\ref{sec:equid-pseud-semigr} revisits equidivisible
pseudovarieties of
semigroups~\cite{Almeida&ACosta:2017,Almeida&ACosta:2023}, extending
them to semigroupoids; this pops up naturally from the study of
concatenation-closed pseudovarieties, and sets the stage for
Theorem~\ref{t:right-stabilizers-L-chain}, which is critically important for our paper in preparation~\cite{Almeida&ACosta&Goulet-Ouellet:2024b}.
Section~\ref{sec:recurrence} is a study of pseudopaths and their prefixes, with an emphasis on prefix accessible pseudopaths. Section~\ref{sec:stabilizers} closes the paper
with results about stabilizers.

\section{Basic definitions}
\label{sec:prelims}

\subsection{Graphs}

In this paper, a \emph{graph} $G$ is a set equipped with a partition into a set $V(G)$ of \emph{vertices} and a set $E(G)$ of \emph{edges}, and endowed with two mappings $\init,\term\from E(G)\to V(G)$. 
The mappings $\init$ and $\term$ are called
the \emph{adjacency} mappings, with
$\init(g)$ and $\term(g)$ respectively being the \emph{source} and the
\emph{range} of an edge~$g$. Set
$G(x,y)=\init^{-1}(x)\cap\term^{-1}(y)$, and $G(x)=G(x,x)$, for
$x,y\in V(G)$, and call $G(x,y)$ a \emph{hom-set} and $G(x)$ a
\emph{local set}. Edges in the same hom-set are \emph{coterminal}. If
we endow $G$ with a Hausdorff topology such that $V(G)$ and $E(G)$ are closed
and the adjacency mappings are continuous, then we say that $G$ is a \emph{topological
graph}. If this topology is compact, then $G$ is a \emph{compact graph}.
We also include the Hausdorff property in the definition of compact space.

With the definition of graph given above, the notions of
subgraphs, morphisms of graphs, and products of graphs follow naturally,
cf.~\cite{Tilson:1987,Almeida&Weil:1996,Jones:1996}. 
More precisely, a \emph{morphism of graphs} $\varphi\from G\to H$ is a map such that
$\varphi(V(G))\subseteq V(H)$ and 
$\varphi(G(x,y))\subseteq H(\varphi(x),\varphi(y))$ for all $x,y\in V(G)$.
A morphism
$\varphi\from G\to H$ is \emph{faithful} if the restriction
$\varphi\from G(x,y)\to H(\varphi(x),\varphi(y))$ is injective for
all $x,y\in V(G)$; it is a \emph{quotient} if its restriction $V(G)\to
V(H)$ is bijective and all its restrictions $G(x,y)\to
H(\varphi(x),\varphi(y))$ are onto. A \emph{graph
  equivalence} on $G$ is an equivalence relation $\theta$
on $G$ that does not identify distinct vertices and only identifies distinct edges that are coterminal.
The natural mapping $q_\theta\from G\to G/{\theta}$
induces in $G/{\theta}$ a graph structure for which $q_\theta$ is a
quotient morphism.

In this paper, we use the term \emph{alphabet} as a synonym for \emph{set}. 
When the context is clear, we may see an alphabet $A$ as a one-vertex graph whose edges are the elements of $A$. For example, the empty set is then seen as the one-vertex graph with no edges. Likewise a topological space may be viewed as a one-vertex topological graph.

\subsection{Semigroupoids}

For a graph $S$,
the set of pairs of \emph{composable egdes}
is the set 
\[
    D(S)=\{(s,t)\in E(S)\times E(S):\init(s)=\term(t)\}.
\]
A \emph{semigroupoid} is a graph $S$ endowed with a mapping $m\from D(S)\to E(S)$,
called the \emph{composition} or \emph{multiplication} of $S$, such that, denoting $m(s,t)$ by $st$, one has:
\begin{enumerate}
    \item $\init(st)=\init(t)$ and $\term(st)=\term(s)$ for every $(s,t)\in D(S)$;
    \item if $(s,t)\in D(S)$ and $(t,r)\in D(S)$ then $(st)r=s(tr)$.
\end{enumerate}
A \emph{local identity} at vertex $x\in V(S)$
is an edge $e\in S(x)$ such that $es=te$ for all $s\in\term^{-1}(x)$ and $t\in\init^{-1}(x)$.
The semigroupoids where each vertex has a (necessarily unique) local identity are precisely the small categories.
We denote by $S^I$ the small category obtained from the semigroupoid $S$ by adjoining
at each vertex $v$ a local identity $1_v$ that is not in $S$.

\begin{remark}
  Frequently in the literature on Semigroup Theory
  (e.g.~\cite{Tilson:1987,Almeida&Weil:1996,Jones:1996}), a pair of edges $(s,t)\in E(S)\times E(S)$ is defined to be composable
  when we have instead $\term(s)=\init(t)$. 
  The definition adopted in this paper is instead consistent with the usual
  convention in Category Theory for composition of morphisms, and is also used in our paper~\cite{Almeida&ACosta&Goulet-Ouellet:2024b}.
\end{remark}

If $S$ is a topological (compact) graph whose multiplication $m$ is continuous, then we say that $S$ is a \emph{topological (compact) semigroupoid}. 
In that case, $S^I$ is also a topological (compact) semigroupoid with the topology defined as follows: first, carry the topology of $V(S)$ to the set $\{1_v \mid v\in V(S)\}$ of new local identities using the bijection $v\mapsto 1_v$;
then equip $E(S^I)$ with the coproduct topology of $E(S)$ with $\{1_v \mid v\in V(S)\}$.
With this topology, $E(S)$ is a clopen subspace of $E(S^I)$ and, in case $V(S)$ is finite, the new local identities are isolated points.

We may see a (topological) semigroup $S$ as a one-vertex (topological)
semigroupoid whose edges are the elements of $S$, the composition
being the semigroup operation (and the topology of the space of edges
being that of~$S$). This is convenient for dealing simultaneously with
semigroups and semigroupoids. It extends the aforementioned way of
seeing sets as one-vertex graphs. Accordingly, we see the empty set
as a semigroup, as done, for instance,
in~\cite{Rhodes&Steinberg:2009qt}, but not in~\cite{Almeida:1994a}.

A subgraph $T$ of a semigroupoid  $S$ is a
\emph{subsemigroupoid}
of $S$ if $T$ is a semigroupoid
whose multiplication is a restriction
of the multiplication of $S$. Given a nonempty
subgraph $X$ of the semigroupoid $S$,
the intersection of all
subsemigroupoids of $S$
containing $X$ is a semigroupoid, called the
\emph{subsemigroupoid of $S$ generated by $X$}.

For semigroupoids $S$ and $T$,
a \emph{homomorphism of semigroupoids}
from $S$ to~$T$ is a morphism of graphs $\varphi:S\to T$
such that $\varphi(s\cdot t)=\varphi(s)\cdot\varphi(t)$
for every $(s,t)\in D(S)$.
If the restriction of $\varphi$ to $V(S)$ is injective, then
$\varphi(S)$ is a subsemigroupoid of~$T$, but that may not be the case otherwise~\cite[Example 3.1]{Almeida&ACosta:2007a}.
A homomorphism $\varphi\from S\to T$ extends to a homomorphism $\varphi^I\from S^I\to T^I$,
such that $\varphi^I(1_v)=1_{\varphi(v)}$ for every $v\in V(S)$.
Note that $\varphi^I$ is continuous if $\varphi$ is a continuous homomorphism of topological semigroupoids. In the absence of confusion, we may denote $\varphi^I$ also by $\varphi$.

The \emph{free semigroupoid generated by the graph $A$}, denoted
$A^+$, is constructed as follows: 
\begin{enumerate}
    \item the set of vertices is given by $V(A^+)=V(A)$; 
    \item for every $x,y\in V(A)$, the set $A^+(x,y)$ consists of the finite sequences $a_0\cdots
a_{n-1}$, with $n\geq 1$ such that $(a_i,a_{i+1})\in D(A)$,
$\term(a_0)=y$ and $\init(a_{n-1})=x$;
    \item for $(s,t)\in D(A^+)$, edge composition is given by concatenation, that is, $(s,t)\mapsto st$.
\end{enumerate}
The small category $(A^+)^I$ is denoted $A^*$. The edges of $A^+$ are the
\emph{paths} over $A$, with the local identities of $A^*$ being the
\emph{empty paths}. For the path $a_0\cdots a_{n-1}$ as above, the
number $n$ is its \emph{length}, and the length of empty paths is
zero. A \emph{language} over the graph $A$ is a subset of $E(A^+)$.
If $A$ is a set (viewed as a one-vertex graph) then $A^+$ and $A^*$ are simply the free semigroup and the
free monoid generated by $A$, respectively.

We adopt the following conventions for subgraphs $X,Y$ of a semigroupoid $S$:
\begin{itemize}
    \item the least subgraph of $S$ containing $\{xy:(x,y)\in D(S)\cap E(X)\times E(Y)\}$ is denoted $XY$;
    \item $X^+$ denotes the subsemigroupoid of $X$ generated by $S$, when no confusion arises with the free semigroupoid generated by $X$.
\end{itemize}

Ideals and Green's relations in semigroups are classical and can be recalled in many standard textbooks~\cite{Clifford&Preston:1961,Almeida:1994a,Grillet:1995bk,Howie:1976,Lallement:1979}. Ideals and Green's relations in a semigroupoid $S$ are defined as in
semigroups. For example, a subgraph $J$ is an ideal when $S^IJS^I=J$. We also let
$x\leq_{\green{R}}y$ when $xS^I\subseteq yS^I$, in which case we say that $y$ is a \emph{prefix} of $x$. We further write $x\green{R}y$ when
$xS^I=yS^I$. Dually, one has the relations $\leq_{\green{L}}$ and
$\green{L}$, and the notion of \emph{suffix}.

\subsection{Pseudovarieties}

A semigroupoid $S$ is a \emph{divisor} of a semigroupoid $T$ if there
is a semigroupoid $R$ for which there are a faithful homomorphism
$R\to T$ and a quotient homomorphism $R\to S$. In particular, a
semigroup $S$ is a divisor of $T$ if and only if it is a subsemigroup of
a homomorphic image of $T$.

A \emph{pseudovariety of semigroups (respectively, semigroupoids)} is
a class of finite semigroups (respectively, semigroupoids) closed
under taking divisors, finite (possibly empty) direct products, and,
in the case of semigroupoids, finite coproducts. The
pseudovariety of all finite semigroups is denoted $\pv{S}$, while that of
all finite semigroupoids is denoted $\pv{Sd}$. As the empty product is
allowed, all pseudovarieties contain the trivial semigroup.

The intersection of a family of pseudovarieties of semigroupoids (semigroups) is a pseudovariety of semigroupoids (semigroups).
For a pseudovariety of semigroups~$\pv V$, the
intersection of all pseudovarieties of semigroupoids containing $\pv V$
is called the  \emph{global} of $\pv V$, denoted by~$\glob \pv V$;
another related pseudovariety of semigroupoids is the class $\ell\pv V$, called the \emph{local of $\pv V$},
of semigroupoids whose local sets are semigroups from $\pv V$.
If $\pv W$ is a pseudovariety of semigroupoids, then  $\pv V=\pv S\cap\pv W$ is a pseudovariety of semigroups
and the inclusions $\glob \pv V\subseteq \pv W\subseteq\ell\pv V$ hold.
A pseudovariety of semigroups is \emph{local} if $\glob \pv V=\ell \pv V$. See~\cite{Almeida:1996c,Tilson:1987,Almeida&Weil:1996} for
introductions to this notion, explanations for its motivation, and
examples of local and non-local pseudovarieties.

\subsection{Profinite semigroupoids}

In what follows:
\begin{itemize}
    \item a \emph{finite-vertex graph (semigroupoid)} means a graph (semigroupoid) with only a finite number of vertices;
    \item finite graphs and semigroupoids are viewed as topological graphs and topological semigroupoids endowed with the discrete topology;
    \item a \emph{quotient inverse limit} of semigroupoids is an inverse limit
    of semigroupoids where every connecting homomorphism is a quotient homomorphism.
\end{itemize}

Let $\Cl C$ be a class of compact semigroupoids.
A compact semigroupoid is \emph{pro-$\Cl C$} if it is an inverse limit of members of $\Cl C$.
We use the terms \emph{profinite semigroupoid} and \emph{profinite semigroup} as synonyms for pro-$\pv{Sd}$ and pro-$\pv S$, respectively.
A compact semigroupoid~$S$ is \emph{residually $\Cl C$} if, for every $u,v\in S$ such that $u\neq v$, there is a continuous homomorphism $\varphi\from S\to F$, with $F\in\Cl C$, such that $\varphi(u)\neq\varphi(v)$.

The next proposition shows that the two properties of being pro-$\Cl C$ and being residually $\Cl C$ are closely related. It will be used freely throughout the paper. It is encapsulated in~\cite[Theorem 5.1]{Jones:1996}, where the hypothesis that the semigroupoid is finite-vertex is implicit (see also~\cite[Theorem 4.1]{Jones:1996}). Without this hypothesis,
the proposition fails (see the comment following~\cite[Theorem 3.6]{Almeida&ACosta:2007a}). The aforementioned theorem~\cite[Theorem 5.1]{Jones:1996} is stated for (small) categories instead of semigroupoids, but, as mentioned in the last section of the same paper, the statements and arguments hold \emph{mutatis mutandis} for semigroupoids.  
  
  \begin{proposition}
    Let $S$ be a finite-vertex compact semigroupoid
    and $\pv V$ be a pseudovariety of semigroupoids.
    The following conditions are equivalent:
    \begin{enumerate}
    \item $S$ is pro-$\pv V$.
    \item $S$ is a quotient inverse limit of pro-$\pv V$ semigroups.
    \item $S$ is residually \pv V.
    \item For every $u,v\in S$ such that $u\neq v$, there
      is a continuous quotient homomorphism $\varphi\from S\to F$, with $F\in\pv V$, such that $\varphi(u)\neq\varphi(v)$.
    \end{enumerate}
  \end{proposition}

  The next proposition appears in a different form as part of~\cite[Proposition 10.2]{Jones:1996}.
  Alternatively, the arguments used in the proof of~\cite[Proposition 3.5]{Almeida:2003cshort}, for the special case of semigroups,
  extrapolate straightforwardly to semigroupoids.
  
  \begin{proposition}
    \label{p:recognition-pro-v-semigroupoid}
    For a pseudovariety of semigroupoids $\pv V$, let $S$ be a finite-vertex pro-$\pv V$ semigroupoid and $K\subseteq S$. Then $K$ is clopen if and only if
    there is a continuous quotient homomorphism $\varphi\from S\to T$ such that $T\in\pv V$ and $K=\varphi^{-1}\varphi(K)$.
 \end{proposition}

 \section{Relatively free profinite semigroupoids}
 \label{sec:relat-free-prof}

 In this section, we recall the definition of relatively free profinite semigroupoid and establish some basic facts concerning them; see \cite{Jones:1996} for further details.
Let $A$ be a finite-vertex graph. 
Take a pseudovariety of semigroupoids~$\pv V$. 
A \emph{free pro-$\pv V$ semigroupoid over $A$} is a pair $(F,\iota)$ of a pro-$\pv V$ semigroupoid $F$ and a continuous graph morphism $\iota\from A\to F$ with the following universal property:
for every continuous graph morphism $\varphi\from A\to S$ where $S$ is a pro-$\pv V$ semigroupoid, there is a unique continuous homomorphism
$\pro\varphi\from F\to S$ such that $\pro\varphi\circ \iota=\varphi$ (see Diagram~\ref{eq:universal-property-free}). We also say that $(F,\iota)$ is a \emph{relatively free profinite semigroupoid}.
\begin{equation}\label{eq:universal-property-free}
\begin{split}
    \xymatrix@C=10mm@R=6mm{
      A \ar[r]^{\iota} \ar[dr]_\varphi & F \ar[d]^{\pro{\varphi}}\\
       & S
    }
\end{split}
\end{equation}

It turns out that there is indeed a free pro-$\pv V$ semigroupoid over $A$, denoted $(\Om AV,\iota_A)$, 
which moreover (by standard categorical arguments) is unique up to isomorphism.
In view of this uniqueness, we may speak about \emph{the} free pro-$\pv V$ semigroupoid over $A$. 
The subsemigroupoid of $\Om AV$ generated by $\iota_A(A)$ is denoted $\om AV$ and is dense in $\Om AV$.

From the definition of $\Om AV$, it follows easily that the restriction of $\iota_A$ to $V(A)$ is injective. 
Moreover $\iota_A$ is injective whenever $\pv V$ contains a semigroup with at least two elements, in which case we view $A$ as a subgraph of $\Om AV$ and $\iota_A$ as the inclusion mapping.

\begin{remark}
  The hypothesis that $A$ is finite-vertex is used in the construction of $\Om AV$.
  See~\cite{Almeida&ACosta:2007a} for the difficulties arising when dealing with infinite-vertex graphs.
\end{remark}

We have therefore a functor $F_{\pv V}$ from the category of finite-vertex graphs to the category of pro-$\pv V$ semigroupoids, given by $F_{\pv V}(A)=\Om AV$ for every finite-vertex graph $A$, and mapping each graph morphism $\varphi\from A\to B$ to the unique continuous homomorphism $F_{\pv V}(\varphi)\from \Om AV\to\Om BV$ such that Diagram~\eqref{eq:universal-property-functor}
commutes. In the absence of confusion, we may denote $F_{\pv V}(\varphi)$ simply by $\pro{\varphi}$.
\begin{equation}\label{eq:universal-property-functor}
\begin{split}
    \xymatrix@C=10mm@R=6mm{
      A \ar[r]^{\iota_{A}} \ar[d]_\varphi & \Om AV \ar[d]^{F_{\pv V}(\varphi)=\pro{\varphi}}\\
      B \ar[r]^{\iota_{B}} & \Om BV
    }
\end{split}
  \end{equation}

This discussion about free pro-$\pv V$ semigroupoids carries on, \emph{mutatis mutandis},
for every pseudovariety of semigroups $\pv V$ and every alphabet $A$, entailing
the existence of the free \emph{free pro-\pv{V} semigroup over $A$}. 

\begin{remark}
  \label{r:om-global-V-alphabet}
  When $\pv V$ is a pseudovariety of semigroups and $A$ is an alphabet, one has $\Om AV=\Om A{\glob V}$.
\end{remark}

Bear in mind Remark~\ref{r:om-global-V-alphabet}, as it allows in many instances
an immediate passage from the context of semigroupoids to that of semigroups.

\begin{remark}
  The notation $\Om AV$, due to Reiterman~\cite{Reiterman:1982}, has a strong presence in the literature and is related with the interpretation
  of the elements of $\Om AV$ as a certain kind of natural transformations called \emph{implicit operations}~\cite{Reiterman:1982,Almeida:1994a,Almeida&Weil:1996}. 
\end{remark}

We proceed to highlight some properties of relatively free profinite semigroupoids.

\begin{lemma}
  \label{l:factorization-of-finite-projection}
  Let $A$ be a finite-vertex graph and $\pv V$ be a pseudovariety of
  semigroupoids. If $\varphi\from \Om AV\to S$ is as continuous
  homomorphism with $S\in\pv V$, then there is a finite-index graph
  equivalence $\theta$ of $A$ and a continuous homomorphism $\psi\from
  \Om {A/{\theta}}V\to S$ such that $\varphi=\psi\circ \pro{q}_\theta$.
\end{lemma}

\begin{proof}
  Let $\theta$ be the graph equivalence relation on $A$ such that, for
  every pair of coterminal edges $a,b$ of $A$, one has $a\mathrel{\theta}b$ if and only if
  $\varphi\circ\iota_A(a)=\varphi\circ\iota_A(b)$. Since $S$ is
  finite, the index of $\theta$ is finite. Set $B=A/{\theta}$. Then
  $\varphi\circ\iota_A$ factors as $\varphi\circ\iota_A=\pro{\psi}\circ
  q_\theta$ for a unique graph morphism $\psi\from B\to S$. Since
  $S\in\pv V$, there is a unique continuous homomorphism $\pro{\psi}\from
  \Om BV\to S$ such $\pro{\psi}\circ\iota_B=\psi$. We then have the following
  commutative diagram
  \begin{equation*}
    \xymatrix@R=5mm@C=5mm{
      A\ar[rr]^{\iota_A}\ar[dd]_{q_{\theta}}&&\Om AV\ar[dl]^\varphi\ar[dd]^{\pro{q}_\theta}\\
      &S&\\
      B\ar[ur]^{\psi}\ar[rr]_{\iota_B}&&\Om BV\ar[ul]_{\pro{\psi}}
    }
  \end{equation*}
  and so $\pro{\psi}\circ \pro{q_\theta}\circ\iota_A=\varphi\circ\iota_A$.
  As $\pro{\psi}\circ \pro{q_\theta}$ and $\varphi$ are both continuous homomorphisms
  and $\iota_A$ generates the topological semigroupoid $\Om AV$,
  it follows that $\pro{\psi}\circ \pro{q_\theta}=\varphi$.
\end{proof}

\begin{corollary}
  \label{c:separation-of-points}
  Let $A$ be a finite-vertex graph and $\pv V$ be a pseudovariety of
  semigroupoids. Let $u,v\in\Om AV$ be such that $u\neq v$.
  Then there is a finite-index graph equivalence $\theta$ of $A$ such
  that $\pro{q}_\theta(u)\neq \pro{q}_\theta(v)$.
\end{corollary}

\begin{proof}
  As $\Om AV$ is pro-$\pv V$, there is a continuous homomorphism $\varphi\from \Om AV\to S$, with $S\in\pv V$,
  such that $\varphi(u)\neq\varphi(v)$. We then apply Lemma~\ref{l:factorization-of-finite-projection}
  to~$\varphi$.
\end{proof}

Let $A$ be a finite-vertex graph.
Then we may consider the directed set $\Theta_A$ of graph equivalences of $A$ with finite index, ordered by reverse inclusion.
For every $\rho,\theta\in \Theta$ such that $\rho\supseteq \theta$, let $q_{\rho/{\theta}}$
be the natural mapping $A/{\theta}\to A/{\rho}$.
Then $\Theta$ is a directed set, and, by functoriality of $F_{\pv V}$,
Diagram~\ref{eq:inverse-limit-finite-retracts} commutes and we may consider the inverse limit $\varprojlim_{\theta\in\Theta}\Om {A/{\theta}}V$
where the homomorphisms of the form $\pro{q}_{\rho/{\theta}}$ are the connecting morphisms.
Hence, there is a continuous homomorphism
  $q\from \Om AV\to \varprojlim_{\theta\in\Theta}\Om {A/{\theta}}V$ such that
  $q(u)=(\pro{q}_\theta(u))_{\theta\in\Theta}$ for every $u\in \Om AV$.

  \begin{equation}\label{eq:inverse-limit-finite-retracts}
    \begin{split}
    \xymatrix@R=0.3 cm@C=1.5 cm{
      &\Om AV\ar[dr]^{\pro{q}_\theta}\ar[dl]_{\pro{q}_\rho}&\\
      \Om {A/{\rho}}V
      &
      &
      \Om {A/{\theta}}V. \ar[ll]^{\pro{q}_{\rho/{\theta}}}
        }
    \end{split}
  \end{equation}

\begin{corollary}
  \label{c:inverse-limit-finite-retracts}
  Let $A$ be finite-vertex graph and $\pv V$ be a pseudovariety of semigroupoids $\pv V$.
  Then $q\from \Om AV\to \varprojlim_{\theta\in\Theta}\Om {A/{\theta}}V$
  is a continuous isomorphism.
\end{corollary}

\begin{proof}
  By Corollary~\ref{c:separation-of-points},
  the mapping $q$ is injective. Since $\pro{q}_\theta$ is onto for every
  $\theta\in\Theta$, by compactness it follows that $q$ is onto~\cite[Section 3.2]{Engelking:1989}, whence an isomorphism.
\end{proof}

Let $A$ be a graph. A \emph{retraction} of $A$ is an onto graph morphism $r\from A\to B$
for which there is a graph morphism $s\from B\to A$ such that $r\circ s=id_B$, in which case we say that $B$ is a \emph{retract} of $A$ and that $s$ is a \emph{section} (of $r$, or of $B$).

\begin{example}
  If $\theta$ is a graph equivalence on $A$, then the graph morphism
  $q_\theta\from A\to A/{\theta}$ is a retraction, with any graph homomorphism
  $s\from A/{\theta}\to A$ such that $s(a/\theta)\in a/\theta$ for
  every $a\in E(A)$ being a section of $r$.
\end{example}

Let $B$ be a subgraph of the graph $A$. We say that $B$ is a
\emph{retract subgraph} of $A$ if there is a retraction $r\from A\to
B$ admitting the inclusion $i\from B\to A$ as a section.

\begin{remark}
  \label{r:embedding-of-retract-subgraph}
  Let $B$ be a retract subgraph of a finite-vertex graph $A$. For the
  inclusion mapping $i\from B\to A$, we may then take a retraction
  $r\from A\to B$ such that $r\circ i=id_B$. By functoriality, the
  equality $\pro{r}\circ \pro{i}=id_{\Om BV}$ holds, whence the
  continuous homomorphism $\pro{i}\from \Om BV\to \Om AV$ is
  injective.
\end{remark}

If $i\from B\to A$ is not a retraction, then $\pro{i}$ may not be injective, as seen next.

\begin{example}
  Let $A$ be the graph consisting of two vertices $x,y$ and three edges
  $a,b,c$ such that $a,b\in A(x,y)$ and $c\in A(y,y)$. The subgraph
  $B$ obtained from $A$ by deleting $c$ is not a retract of $A$, 
  since any homomorphic image of $A$ must contain a loop. 
  On the other hand, let $\pv V$ be the class of all finite semigroupoids
  $S$ such that, for any triple of edges $(s,t,r)$ of $S$ satisfying
  $\init(s)=\init(t)$ and $\term(s)=\term(t)=\init(r)=\term(r)$, the
  equality $s=t$ holds. As the class $\pv V$ is really defined by a
  pseudoidentity in the sense of~\cite{Tilson:1987}, it is a
  pseudovariety of semigroupoids. Then $\Om BV=B$, while the graph
  $\Om AV$ is the quotient of $A$ obtained by identifying $a$ with
  $b$. Hence, it is impossible to find an injective graph homomorphism 
  from $\Om BV$ to $\Om AV$.
  \end{example}

When $B$ is a retract subgraph of a finite-vertex graph~$A$, we may see
$\Om BV$ as a closed subsemigroupoid of $\Om AV$ by means of
Remark~\ref{r:embedding-of-retract-subgraph}. If $\pv V$ contains
the pseudovariety $\pv{Sl}$ of all finite semilattices, then we have
the following stronger property.

\begin{proposition}
  \label{p:over-Sl}
  Let \pv V be a pseudovariety of semigroupoids, $A$ a finite-vertex graph, and
  $B$ a retract subgraph of $A$.
  If\/ $\pv V$ contains $\pv {Sl}$, then $\Om BV$ is open in $\Om AV$.

\end{proposition}

\begin{proof}
  Consider the two-element semilattice $T=\{0,1\}$, where $0$ is a zero.
  Let $\chi\from A\to T$ be the graph morphism
  defined by
  \begin{equation*}
    \chi(a)=
    \begin{cases}
      1&\text{if $a\in E(B)$},\\
      0&\text{if $a\in E(A)\setminus E(B)$}.
    \end{cases}
  \end{equation*}  
  As $\pv V\supseteq \pv{Sl}$, we may consider
  the unique continuous homomorphism $\hat{\chi}\from\Om AV\to T$
  extending $\chi$. We claim that $\hat{\chi}^{-1}(1)=\Om BV$, which shows that $\Om BV$~is open. The inclusion $\Om BV\subseteq\hat{\chi}^{-1}(1)$ follows
  from $\hat{\chi}^{-1}(1)$ being a closed subsemigroupoid of~$\Om AV$ containing~$B$.
  For the reverse inclusion, note that
  every edge of $\om AV$ having an element of $E(A)\setminus E(B)$ as
  factor is mapped by $\chi$ to $0$,
  whence $\om AV\cap\hat{\chi}^{-1}(1)\subseteq \om BV$;
  as $\hat{\chi}^{-1}(1)$ is open and $\om AV$ is dense in $\Om AV$,
  it follows that $\hat{\chi}^{-1}(1)\subseteq\Om BV$.
\end{proof}

\begin{remark}
    In the special case where $A$ is a set viewed as a one-vertex graph, any subset of $A$ is a retract. By Remark~\ref{r:om-global-V-alphabet}, it immediately follows from the previous result that if $\pv V$ is a pseudovariety of semigroups containing $\pv{Sl}$ and $A$ is any alphabet, then $\Om BV$ is an open subsemigroup of $\Om AV$ for every $B\subseteq A$.
\end{remark}

Moving forward, we focus on pseudovarieties containing the pseudovariety $\pv N$ of
all finite nilpotent semigroups. If $A$ is an alphabet, then $\Om AN$ is
the one-point compactification of the discrete space $A^+$, the
semigroup structure of $\Om AN$ extending $A^+$ by letting the point
at infinite be a zero~\cite[Section~3.7]{Almeida:1994a}. If $A$ is
singleton, then $A^+$ is isomorphic to the additive semigroup $\nn_+$
of positive integers, and $\Om AN$ becomes the one-point
compactification $\nn_+\cup\{\infty\}$, extending $\nn_+$ with
$\infty$ as a zero for the semigroup operation.

\begin{example}
  \label{ex:free-nilpotent-retract-subgraph}
    Let $A$ be an alphabet with at least two letters. Then the free pro-\pv N semigroup $\Om AN$ contains one isomorphic copy of $\nn_+\cup\{\infty\}$ for each letter $a\in A$, since $\{a\}\subseteq A$ is a retract subgraph when $\{a\}$ and $A$ are viewed as one-vertex graphs. All of these copies of $\nn_+\cup\{\infty\}$ intersect pairwise in the single point $\{\infty\}$, and thus they are closed but not open in $\Om AN$. This shows that the conclusion of Proposition~\ref{p:over-Sl} fails when $\pv V = \pv N$.
\end{example}

A proof of the next proposition may be found in~\cite[Proposition 3.12]{Almeida&ACosta:2007a}.

\begin{proposition}
  \label{p:A+-embeds-isolated}
    Let $A$ be a finite-vertex graph and \pv V be a pseudovariety of
    semigroupoids containing \pv N. The unique homomorphism $A^+\to \Om AV$
    extending the mapping $\iota_A\from A\to\Om AV$ is injective, and the elements in its image are isolated points of $\Om AV$.
  \end{proposition}

  For a subset $X$ of $\Om AV$, we use the notation
 $\clos V(X)$ for the topological closure of $X$ in $\Om AV$.
 In view of Proposition~\ref{p:A+-embeds-isolated}, when $\pv V$ is a
 pseudovariety of semigroupoids containing~$\pv N$, we regard $A^+$ as
 a subsemigroupoid of $\Om AV$ which is a discrete dense subspace of
 $\Om AV$. Hence, for such $\pv V$ and $A$, the following observation
 arises: for every language $L\subseteq A^+$, the equality $\clos
 V(L)\cap A^+=L$ holds. Proposition~\ref{p:A+-embeds-isolated} and
 these considerations motivate calling edges of $\Om AV$
 \emph{pseudopaths}.

  We say that a subgraph $F$ of a semigroupoid $S$ is \emph{factorial in $S$}
  if the inclusion $m^{-1}(F)\subseteq F\times F$ holds, where $m\from S\times S\to S$ denotes the multiplication mapping of $S$.

\begin{corollary}
  \label{c:A+-as-filter}
   Let $A$ be a finite-vertex graph and \pv V be a pseudovariety of
   semigroupoids containing \pv N. Then  $A^+$ is factorial in $\Om AV$.
\end{corollary}

\begin{proof}
  Take $u\in E(A^+)$.
  Let $m$ be the multiplication mapping
  of $\Om AV$. 
  By Proposition~\ref{p:A+-embeds-isolated},
  the set $\{u\}$ is clopen. Hence, as~$m$ is continuous and $D(\Om
  AV)$ is clopen in $\Om AV$ because $A$~ is a finite-vertex graph,
  the set $m^{-1}(u)$ is clopen in $\Om AV\times\Om AV$.
  As $A^+\times A^+$ is dense in $\Om AV\times\Om AV$,
  it follows that $m^{-1}(u)$ is the topological closure
  of the intersection $m^{-1}(u)\cap A^+\times A^+$.
  Such intersection is finite, thus closed, whence $m^{-1}(u)\subseteq A^+\times A^+$.
\end{proof}

Let $\pv V$ be a pseudovariety of semigroupoids containing $\pv N$ and
$A$ be a finite-vertex graph. Since $\nn_+\cup\{\infty\}$ is pro-$\pv
V$, we may consider the unique continuous homomorphism $\ell_A\from
\Om AV\to \nn_+\cup\{\infty\}$ such that $\ell_A(a)=1$ for all $a\in
E(A)$. Note that if $u$ is a path over~$A$, then~$\ell_A(u)=|u|$. More
generally, for each pseudopath $u\in\Om AV$, we say that $\ell_A(u)$
is the \emph{length} of $u$, and use the notation $\ell_A(u)=|u|$. We
extend this to empty pseudopaths of $(\Om AV)^I$ by letting
$\ell_A(u)=0$ whenever $u$ is an empty pseudopath. Moreover, we say
that a pseudopath $u\in(\Om AV)^I$ has finite length, or is a
\emph{finite-length pseudopath}, if $\ell_A(u)\in\nn$; otherwise, that
is, when $\ell_A(u)=\infty$, we say that it has infinite length, or
that it is an \emph{infinite-length pseudopath}.

\begin{remark}
  \label{r:infinite-length-pseudowords-form-ideal}
  As $\infty$ is a zero of $\nn_+\cup\{\infty\}$,
  the infinite-length pseudopaths of $\Om AV$ form an ideal of $\Om AV$,
  whenever $\pv N\subseteq \pv V$, for any finite-vertex graph $A$.
    In other words, the set of finite-length pseudopaths is factorial.
\end{remark}

In the case of finite graphs, we have the following stronger property.

\begin{proposition}
  \label{p:finite-length-pseudopaths-in-finite-graph-case}
  If\/ $\pv V$ is a pseudovariety of semigroupoids containing $\pv N$ and $A$ is a finite graph, then a pseudopath
  of\/ $\Om AV$
  has finite length if and only if it is in $A^+$.
\end{proposition}

\begin{proof}
  The ``if'' part is trivial. Conversely, take a pseudopath $u\in\Om AV$ not in $A^+$.
  As $A^+$ is dense in $\Om AV$, there is a net $(u_i)_{i\in I}$
  of elements of $E(A^+)$ converging to $u$. Since $u\notin A^+$, such net has infinitely many distinct values. Hence, as $A$ is finite, the net
  $(\ell_A(u_i))_{i\in I}$ is unbounded. By continuity of $\ell_A$, it follows that $\ell_A(u)=\infty$.
\end{proof} 

\begin{remark}
  \label{r:pseudo-letters}
  The assumption that $A$ is a finite graph cannot be dropped in the previous proposition.
  If\/ $\pv V$ is a pseudovariety of semigroupoids containing $\pv N$ and $A$ is any finite-vertex graph, then $\clos V(A)=\ell_A^{-1}(1)$.
  In particular, as the elements of $A$ are isolated, if $A$ is infinite then there are pseudowords of length~$1$ not in $A^+$.
\end{remark}

What we saw for pseudovarieties of semigroupoids containing $\pv N$ carries on \emph{mutatis mutandis}
for pseudovarieties of semigroups containing $\pv N$. For such a pseudovariety of semigroups $\pv V$, and an alphabet $A$,
the pseudopaths of $\Om AV$ are called \emph{pseudowords}, as they generalize the words in the free semigroup $A^+$.

Given a finite-vertex graph $A$, let $\gamma_A^{\pv V}$, or simply $\gamma$, be the unique continuous homomorphism $\Om A{\glob V}\to \Om {E(A)}V$
extending the graph morphism $A\to E(A)$ which collapses all vertices
and restricts to the identity on $E(A)$.

\begin{proposition}
  \label{p:faithfulness}
  Let $A$ be a finite-vertex graph, and $\pv V$ be a pseudovariety of semigroups. Then $\gamma\from \Om A{\glob V}\to\Om {E(A)}V$ is faithful.
\end{proposition}

Proposition~\ref{p:faithfulness} is shown in~\cite[Proposition~2.3]{Almeida:1996c}
and~\cite[Theorem~8.4]{Jones:1996}. In the former paper, it is stated under the assumption that $A$ is finite, and in both papers in the setting of \emph{pseudovarieties of categories}.

For a set $Q$, consider the \emph{(aperiodic) Brandt semigroup} $B_Q$:
the underlying set of $B_Q$ is $(Q\times Q)\cup\{0\}$, with $0$ an
element not in $Q$ which is a zero for $B_Q$, and multiplication of
elements of $Q\times Q$ given by $(p,r)\cdot (s,q)=0$ if $r\neq s$ and
$(p,r)\cdot (r,q)=(p,q)$. If $Q$ and $R$ have finite cardinal $n$,
then $B_Q\cong B_R$, justifying the notation $B_n$ for a
representative of the isomorphism class of $B_Q$. It is well known
that, for a pseudovariety of semigroups $\pv V$, one has $B_2\in\pv V$
if and only if $B_Q\in\pv V$ for every finite set $Q$. Parts of the
following proposition are spread in~\cite{Jones:1996}.

\begin{proposition}
  \label{p:faithful-open-case-B2}
  Let $A$ be a finite-vertex graph, and $\pv V$ be a pseudovariety of semigroups containing $B_2$.
  Then  the following properties hold.
  \begin{enumerate}
  \item The image $\gamma(\Om A{\glob V})$ is an open factorial subset of $\Om {E(A)}V$.\label{item:faithful-open-case-B2-1}
  \item The homomorphism $\gamma$ is an open mapping whose restriction to $E(\Om A{\glob V})$ is a topological embedding.\label{item:faithful-open-case-B2-2}
  \item For all $x,y,z\in E(\Om A{\glob V})$, if the equality $\gamma(x)\cdot \gamma(y)=\gamma(z)$ holds,
  then we have $(x,y)\in D(\Om A{\glob V})$ and $xy=z$.\label{item:faithful-open-case-B2-3}
  \end{enumerate}
\end{proposition}

\begin{proof}
  Let us denote by $\circ$ the composition in $\Om A{\glob V}$.
  Set $Q=V(A)$. Since $B_Q\in\pv V$, there is a unique continuous homomorphism $\varphi\from \Om {E(A)}V\to B_Q$
  such that $\varphi(a)=(\term(a),\init(a))$ for every $a\in E(A)$.

  We claim that the equality 
  \begin{equation}\label{eq:faithful-open-case-B2-1}
    \varphi(\gamma(u))=(\term(u),\init(u))
  \end{equation}
  holds for every edge $u$ of $\Om A{\glob V}$. Suppose first that $u$
  is an edge of $\om A{\glob V}$. Then we have a factorization in $\om
  A{\glob V}$ of the form $u=a_1\circ \cdots \circ a_n$ for some
  $a_1,\ldots,a_n\in A$. By induction on $n$, one obtains
  $\varphi(\gamma(u))=(\term(a_1),\init(a_n))=(\term(u),\init(u))$.
  Hence~\eqref{eq:faithful-open-case-B2-1} holds for every $u\in E(\om
  A{\glob V})$. Since $\om A{\glob V}$ is dense in $\Om A{\glob V}$,
  we then deduce from the continuity of the mappings $\varphi$,
  $\gamma$, $\term$ and $\init$ the validity of the claim.

  In particular, if $x,y\in E(\Om A{\glob V})$ are such that $\gamma(x)=\gamma(y)$, we obtain
  from formula~\eqref{eq:faithful-open-case-B2-1} that $x,y$ are coterminal.
  Since $\gamma$ is faithful by Proposition~\ref{p:faithfulness}, we actually have $x=y$. As $E(\Om A{\glob V})$ and $\Om {E(A)}V$ are compact,
  this shows that $\gamma$ restricts to a topological embedding $E(\Om A{\glob V})\to \Om {E(A)}V$.

  We claim that the equality
  \begin{equation}\label{eq:faithful-open-case-B2-2}
    \gamma(E(\Om A{\glob V}))=\varphi^{-1}(B_Q\setminus\{0\})
  \end{equation}
  holds. By induction on $n\geq 1$, we see that, for all $a_1,\ldots,a_n\in E(A)$, one has $\varphi(a_1a_2\cdots a_{n-1}a_n)\neq 0$
  if and only if the composition  $a_1\circ a_2\circ \cdots \circ
  a_{n-1}\circ a_n$ is well defined in $\om A{\glob V}$.
  This means that $\om {E(A)}V\cap \varphi^{-1}(B_Q\setminus\{0\})=\gamma(E(\om A{\glob V}))$.
  Since $\om {E(A)}V$
  and $E(\om A{\glob V})$ are respectively
  dense in $\Om {E(A)}V$ and $E(\Om A{\glob V})$,
  it then follows from the continuity of $\gamma$ and $\varphi$ that the equality~\eqref{eq:faithful-open-case-B2-2}
  indeed holds.

  Then~\eqref{eq:faithful-open-case-B2-2} implies that $\gamma(E(\Om A{\glob V}))$ is open and factorial in $\Om {E(A)}V$, i.e. property~\ref{item:faithful-open-case-B2-1} is proved.
  Hence, as $\gamma$ restricts to a topological embedding $E(\Om A{\glob V})\to \Om {E(A)}V$,
  the map $\gamma$ is open, which concludes the proof of property~\ref{item:faithful-open-case-B2-2}.

  It only remains to establish property~\ref{item:faithful-open-case-B2-3}.
  Let $x,y,z\in E(\Om A{\glob V})$ be such that $\gamma(x)\cdot \gamma(v)=\gamma(z)$.
  Applying $\varphi$,
  we get from~\eqref{eq:faithful-open-case-B2-1}
  the equality
  \begin{equation*}
    (\term(x),\init(x))\cdot (\term(y),\init(y))=(\term(z),\init(z)),
  \end{equation*}
  which holds in $B_Q$; this forces $\init(x)=\term(y)$, which means that $(x,y)\in D(\Om A{\glob V})$.
  Moreover, from our hypothesis we get $\gamma(x\circ y)=\gamma(z)$.
  Since $\gamma$ is injective on $E(\Om A{\glob V})$, we deduce that $x\circ y=z$.
\end{proof}

  Let $L\subseteq A^+$ be a language over a graph $A$. Let $\Cl C$ be a class of semigroupoids.
  We say that $L$ is \emph{$\Cl C$-recognizable} if and only if there is a semigroupoid
  homomorphism $\varphi\from A^+\to F$, with $F\in\Cl C$, such that $L=\varphi^{-1}(\varphi(L))$.

\begin{remark}
    \label{r:recognizable-languages-trace}
  If $\pv V$ is a pseudovariety of semigroupoids and $A$ is an alphabet, then a language $L\subseteq A^+$ is $\pv V$-recognizable if and only if it is $(\pv S\cap\pv V)$-recognizable.
\end{remark}

The study of profinite semigroups is strongly motivated by the next theorem, established by the first author in the case
of pseudovarieties of semigroups and finite alphabets~\cite[Theorem 3.6.1]{Almeida:1994a}.
The general case for pseudovarieties of finite semigroupoids and arbitrary finite-vertex graphs (possibly with infinitely many edges)
is established in~\cite[Theorems 10.3 and 10.4]{Jones:1996}, again under the guise of pseudovarieties of categories
(cf.~Section~\cite[Section 11]{Jones:1996}).

\begin{theorem}
  \label{t:V-recognizability}
  Let \pv V be a pseudovariety of semigroupoids containing \pv N.
  Let $L$ be a language over a finite-vertex graph $A$.
  The following conditions are equivalent:
  \begin{enumerate}
  \item $L$ is $\pv V$-recognizable;\label{item:V-recognizability-1}
  \item $\clos V(L)$ is open;\label{item:V-recognizability-2}
  \item $L=K\cap E(A^+)$ for some clopen subgraph $K$ of $\Om AV$.
    \label{item:V-recognizability-3}
  \end{enumerate}
\end{theorem}

Using Remark~\ref{r:recognizable-languages-trace}, one can see that the analogous result for semigroups is the special case of one-vertex graphs.

A common operation performed on recognizable languages is that of \emph{concatenation}. In the semigroupoid case, this operation maps two languages $L,K\subseteq E(A^+)$ to the language $LK\subseteq E(A^+)$ consisting of all paths of the form $uv$ for $u\in L$ and $v\in K$ such that $(u,v)\in D(A^+)$. We can deduce the following simple fact from the above theorem.

\begin{corollary}
    \label{c:concatenation-finite-retract-open}
    Let $A$ be a finite-vertex graph and $\pv V$ be a pseudovariety of semigroupoids containing $\pv N$. If the set of $\pv V$-recognizable languages over $A$ is closed under concatenation, then $\Om BV$ is open in $\Om AV$ for every finite retract $B$ of $A$.
\end{corollary}

\begin{proof}
    Let $B$ be a finite retract subgraph of $A$. 
    Clearly $E(A^+)$ is $\pv V$-recognizable,  and so are all singletons $\{a\}$ for $a\in E(A)$ by Proposition~\ref{p:A+-embeds-isolated}. 
    It follows that $E(A^*)aE(A^*)=E(A^+)aE(A^+)\cup E(A^+)a\cup aE(A^+) \cup\{a\}$ is $\pv V$-recognizable for every $a\in E(A)$. Since $E(B)$ is finite, the set  
    \begin{equation*}
    E\bigl((A\setminus B)^+\bigr)=E(A^+)\setminus \bigcup_{b\in E(B)}E(A^*)bE(A^*)
  \end{equation*}
  is also $\pv V$-recognizable. 
  Moreover, the equality
  \begin{equation*}
    E(B^+)=E(A^+)\setminus E\bigl(A^*(A\setminus B)^+A^*\bigr)
  \end{equation*}
  holds, and so, the language $E(B^+)\subseteq A^+$ is $\pv
  V$-recognizable. Since $\Om BV$ is the topological closure of $B^+$
  in $\Om AV$, it follows from Theorem~\ref{t:V-recognizability} that
  $\Om BV$ is open in $\Om AV$.
\end{proof}

We note that the condition of
Corollary~\ref{c:concatenation-finite-retract-open} is sufficient but
not necessary. Indeed, the conclusion holds for all pseudovarieties
containing \pv{Sl} by Proposition~\ref{p:over-Sl}, but not all such
pseudovarieties are such that $\pv V$-recognizable languages are
closed under concatenation.

\section{Open multiplication and concatenation-closed pseudovarieties}
\label{sec:open-multiplication}

In this section we characterize when
multiplication in a relatively free profinite semigroupoid, over a finite-vertex graph, is an open
mapping. We rely on a general, purely topological result
that has nothing to do with algebraic operations.

We say that a mapping between topological spaces is
\emph{open} if it maps open sets to open sets~\cite{Kelley:1975,Engelking:1989,Willard:1970}. Some authors call such
mappings \emph{strongly open}~\cite{Hart&Nagata&Vaughan:2004egt,Whyburn:1964}.
The following result seems to be folklore. The proof presented here is based on an idea of
Kelley~\cite[Lemma 5, page 70]{Kelley:1975}. Nevertheless, we
emphasize that we adopt Willard's definition of subnet
\cite[Definition~11.2]{Willard:1970}, rather than Kelley's. More explicitly, a \emph{subnet} of $(x_i)_{i\in I}$ is a net $(x_{i_j})_{j\in J}$ where the mapping $j\mapsto i_j$ is increasing and cofinal.  

\begin{proposition}
  \label{p:open-mapping}
  Let $f\from X\to Y$ be a mapping between two topological spaces. Then the
  following conditions are equivalent:
  \begin{enumerate}
  \item\label{item:open-mapping-1} the mapping $f$ is open;
  \item\label{item:open-mapping-2} for every open subset $U$ of~$X$,
    every net in~$Y$ converging to some point
    of~$f(U)$ must have points in $f(U)$;
  \item\label{item:open-mapping-3} for every $x\in X$, every net $(y_i)_{i\in I}$ in
    $Y$ converging to~$f(x)$ has a subnet $(y_{i_j})_{j\in J}$ such that $y_{i_j}=f(x_j)$ 
    for all $j\in J$, where $(x_j)_{j\in J}$ is a net in $X$ converging to $x$.
  \end{enumerate}
\end{proposition}

\begin{proof}
  $\ref{item:open-mapping-1}{\Rightarrow}\ref{item:open-mapping-3}$.  
    Suppose first that $f$ is an open mapping and let $x$ be an element
  of~$X$ and $(y_i)_{i\in I}$ be a net in~$Y$ converging to $f(x)$.
    Let $\Cl N_x$ be the set of all neighborhoods of~$x$ and consider
  the set
  \begin{displaymath}
    J = \{(i,N)\in I\times\Cl N_x\from y_i\in f(N)\}
  \end{displaymath}
  ordered by $(i,N)\le(h,M)$ if $i\le h$ and $N\supseteq M$.
  For each $(i,N)\in J$, we choose $x_{(i,N)}\in N$ such that $y_i=f(x_{(i,N)})$.

  Note that, for every $N\in \Cl N_x$, since $f$ is an open mapping,
  $f(N)$ is a neighborhood of~$y=f(x)$. Hence, as $(y_i)_{i\in I}$
  converges to~$y$, there exists $i_N\in I$ such that $i_N\leq i$
  implies $(i,N)\in J$. It follows that the set $J$ is directed and so
  the onto function $\lambda\from J\to I$ given by projection on the
  first component defines a subnet $(y_{\lambda(j)})_{j\in J}$ of
  $(y_i)_{i\in I}$. Note that $y_{\lambda(j)}=f(x_j)$ for all $j\in J$.

  Finally, let $N\in\Cl N_x$. Set $j_0=(i_N,N)$. For every $j=(i,M)\in J$
  we have $x_j\in M$, and so $j_0\leq j$
  implies $x_j\in N$. Hence, the net $(x_j)_{j\in J}$ indeed converges to $x$.
  
  $\ref{item:open-mapping-3}{\Rightarrow}\ref{item:open-mapping-2}$. Assume \ref{item:open-mapping-3} and let $U$ be an open subset of $X$. 
  Let $x\in U$ and let $(y_i)_{i\in I}$ be a net converging to $f(x)$. By \ref{item:open-mapping-3}, there is a subnet $(y_{i_j})_{j\in J}$ of $(y_i)_{i\in I}$ such that $y_{i_j}=f(x_j)$ for all $j\in J$, where $(x_j)_{j\in J}$ is a net converging to $x$. 
  Since $U$ is a neighborhood of $x$, there exists $j_0\in J$ such that $x_j\in U$ for every $j\geq j_0$. 
  In particular $f(x_{j_0})$ is an element of the net $(y_i)_{i\in I}$ which belongs to $f(U)$.
  
  $\ref{item:open-mapping-2}{\Rightarrow}\ref{item:open-mapping-1}$.
  Assume \ref{item:open-mapping-2}. Let $U$ be an open subset of~$X$.
  Let $(y_i)_{i\in I}$ be a net of elements of $Y\setminus f(U)$ converging to an element $y$ of $Y$.
  From \ref{item:open-mapping-2}, we get $y\notin f(U)$. This shows that $Y\setminus f(U)$ is closed.
\end{proof}

\begin{definition}
    Let $S$ be a topological semigroupoid. We say that a convergent net
    $(w_i)_{i\in I}$ in $E(S)$ is \emph{fully factorizable} if, for every factorization
    $\lim w_i=uv$, there is a subnet $(w_{i_j})_{j\in J}$ such that $w_{i_j} = u_jv_j$ for all $j\in J$, where $(u_j,v_j)_{j\in J}$ is a net in
    $D(S)$ converging to $(u,v)$.
\end{definition}

\begin{corollary}
  \label{c:open-mult-fully-factorizable}
  Let $S$ be a topological semigroupoid. Then
  the multiplication in $S$ is open if and only if
  every convergent net in~$S$
  is fully factorizable.
\end{corollary}

\begin{proof}
  It suffices to apply to the multiplication mapping $D(S)\to S$ the equivalence \ref{item:open-mapping-1}$\Leftrightarrow$\ref{item:open-mapping-3}
in Proposition~\ref{p:open-mapping}.
\end{proof}

The following result provides several alternative characterizations of
when multiplication in a relatively free profinite semigroupoid is an
open mapping. 
Most importantly, it relates the open multiplication property with closure of \pv V-recognizable languages under concatenation. 

\begin{theorem}
  \label{t:mult-open-in-free-profinite}
  Let $A$ be a finite-vertex graph and \pv V be a
  pseudovariety of semigroupoids containing~$\pv N$.
  The following conditions are equivalent.
    \begin{enumerate}
        \item\label{item:mult-open-in-free-profinite-1}
        The set of \pv V-recognizable languages over $A$ is closed under concatenation.
        \item\label{item:mult-open-in-free-profinite-2} The multiplication in $\Om AV$ is open.
        \item\label{item:mult-open-in-free-profinite-4} For every finite retract $B$ of~$A$, the multiplication in $\Om BV$ is open.
        \item\label{item:mult-open-in-free-profinite-5} For every retract $B$ of~$A$, the multiplication in $\Om BV$ is open.
    \end{enumerate}
\end{theorem}

\begin{proof}
  \ref{item:mult-open-in-free-profinite-1}$\Rightarrow$\ref{item:mult-open-in-free-profinite-2}.
  Let $K,M$ be clopen subsets of $\Om AV$. 
  By Theorem~\ref{t:V-recognizability}, the languages $L=K\cap A^+$ and
  $N=M\cap A^+$ are $\pv V$-recognizable. Therefore, by hypothesis,
  $LN$ is $\pv V$-recognizable, and so, by Theorem~\ref{t:V-recognizability},
  the set $\clos V(LN)$ is clopen. Since $A$ is dense in $\Om AV$,
  the equalities $L=K\cap A^+$ and
  $N=M\cap A^+$ yield the equalities $\clos V(L)=K$ and $\clos V(N)=M$.
  Therefore, $KM=\clos V(LK)$ is clopen. We showed that
  the product of two clopen sets of $\Om AV$ is always clopen. Since $\Om AV$ is zero-dimensional, this
  establishes that the multiplication in $\Om AV$ is open.
  
  \ref{item:mult-open-in-free-profinite-2}$\Rightarrow$\ref{item:mult-open-in-free-profinite-1}.
  By Theorem~\ref{t:V-recognizability}, if $K,L\subseteq A^+$ are $\pv
  V$-recognizable, then $\clos V(K)$ and $\clos V(L)$ are open subsets
  of $\Om AV$. Hence, by our assumption that the multiplication in
  $\Om AV$ is open, the set $\clos V(K)\clos V(L)$ is open. Since the
  equalities $\clos V(K)\clos V(L)=\clos V(KL)$ and $\clos V(KL)\cap
  A^+=KL$ hold (the former by continuity of multiplication, and the latter
  because $\pv V$ contains $\pv N$), it follows from
  Theorem~\ref{t:V-recognizability} that $KL$ is $\pv V$-recognizable.
  
  \ref{item:mult-open-in-free-profinite-1}$\Rightarrow$\ref{item:mult-open-in-free-profinite-4}.
  By the above, we may assume that \ref{item:mult-open-in-free-profinite-2} also holds.
  Let $B$ be a finite retract of $A$. Every finite retract of $A$ is isomorphic
  to a retract subgraph of~$A$, so we may as well suppose that $B$ is a retract subgraph of $A$. Let $m_B$ and $m_A$ denote the multiplication mappings of $\Om BV$ and $\Om AV$ respectively. It then suffices to observe that the equality $m_B(K)=m_A(K)$ holds for every subset $K$ of $D(\Om BV)$, which implies that $m_B$ is an open mapping since $\Om BV$ is an open subset of $\Om AV$ by Corollary \ref{c:concatenation-finite-retract-open}.
  
  \ref{item:mult-open-in-free-profinite-4}$\Rightarrow$\ref{item:mult-open-in-free-profinite-2}.
  Since $\Om AV$
  is zero-dimensional, it suffices to show that, for arbitrary clopen
  subsets $K$ and $L$, the product $KL$ is open.
  By Proposition~\ref{p:recognition-pro-v-semigroupoid}, we may take a continuous quotient homomorphism $\varphi\from\Om AV\to S$,
  with $S\in\pv V$, such that, for some subsets $P$ and $Q$ of~$S$, one has $K=\varphi^{-1}(P)$ and $L=\varphi^{-1}(Q)$.
    
  By Lemma~\ref{l:factorization-of-finite-projection}, for some finite retract $B$ of $A$
  there are continuous quotient homomorphisms $\eta\from\Om AV\to \Om BV$
  and $\psi\from\Om BV\to S$
  such that $\varphi=\psi\circ\eta$ and $\eta(A)=B$.
   \begin{equation*}
 \begin{split}
    \xymatrix{
      \Om AV\ar[d]_{\eta}\ar[rd]^{\varphi}&\\
      \Om BV\ar[r]_{\psi}&S
       }
   \end{split}                                       
  \end{equation*}
  As $\eta$ is onto, we have
  \begin{equation*}
    \psi^{-1}(P)=\eta\eta^{-1}\psi^{-1}(P)=\eta(\psi\circ\eta)^{-1}(P)=\eta\varphi^{-1}(P)=\eta(K),
  \end{equation*}
  and, similarly, $\psi^{-1}(Q)=\eta(L)$. The equalities $\eta(K)=\psi^{-1}(P)$ and $\eta(L)=\psi^{-1}(Q)$
  yield that $\eta(K)$ and $\eta(L)$ are open subsets of $\Om BV$.
  Noting that $\eta(K)\eta(L)=\eta(KL)$, as $\eta$ is a quotient homomorphism,
  it then follows from our hypothesis that $\eta(KL)$ is also open. Therefore,
  by continuity of $\eta$, to conclude the proof
  it suffices to establish the equality $KL=\eta^{-1}\eta(KL)$.
  
  Let $w\in \eta^{-1}\eta(KL)$. We may consider $u\in \eta(K)$ and
  $v\in\eta(L)$ such that $\eta(w)=uv$. Since $\pv V$ contains $\pv
  N$, we may take advantage of the inclusions $A^+\subseteq\Om AV$ and
  $B^+\subseteq\Om AV$. In particular, as $A^+$ is dense in $\Om AV$,
  we may pick a net $(w_i)_{i\in I}$ of elements of $A^+$ converging
  in $\Om AV$ to $w$. Since $\lim\eta(w_i)=uv$, it follows from
  Corollary~\ref{c:open-mult-fully-factorizable} that there is a net
  $(u_j,v_j)_{j\in J}$ converging in $\Om BV\times\Om BV$ to $(u,v)$
  and a subnet $(\eta(w_{i_j}))_{j\in J}$ of
  $(\eta(w_i))_{i\in I}$ such that
  $\eta(w_{i_j})=u_jv_j$ for all $j\in J$. 
  Recall that $\eta(A)=B$, whence $\eta(A^+)=B^+$. In particular we have
  $u_jv_j=\eta(w_{i_j})\in B^+$ for every $j\in J$. 
  By Corollary~\ref{c:A+-as-filter}, the graph $B^+$ is factorial in $\Om
  BV$. Therefore, $u_j$ and $v_j$ belong to $B^+$ for every $j\in J$.
  Again by the equality $\eta(A)=B$, for each $j\in J$ the paths
  $w_{i_j}$ and $\eta(w_{i_j})$ have the same length,
  and so we may consider a factorization $w_{i_j}=u_j'v_j'$,
  with $u_j',v_j'\in A^+$, such that $\eta(u_j')=u_j$ and
  $\eta(v_j')=v_j$. Let $(u',v')$ be a cluster point of the net
  $(u_j',v_j')$. By continuity of the multiplication of $\Om AV$ and
  of the mapping $\eta$, we respectively get $w=u'v'$ and
  $\eta(u')=u$, $\eta(v')=v$. Since $u\in \eta(K)$, we then have
  $\varphi(u')=\psi(u)\in\psi\eta(K)=P$, and so, as
  $\varphi^{-1}(P)=K$, we see that $u'\in K$. Similarly, one also has
  $v'\in L$. Hence, the equality $w=u'v'$ yields $w\in KL$. This
  establishes the equality $KL=\eta^{-1}\eta(KL)$, which concludes the
  proof that \ref{item:mult-open-in-free-profinite-4} implies \ref{item:mult-open-in-free-profinite-2}.

    \ref{item:mult-open-in-free-profinite-4}$\Rightarrow$\ref{item:mult-open-in-free-profinite-5}.
    Let $B$ be some retract of~$A$.
    Let $C$ be a finite retract of~$B$. Every retract of $B$ is a retract of $A$, and so,
    by the assumption that $\ref{item:mult-open-in-free-profinite-4}$ holds, the multiplication
    in $\Om CV$ is open. We have already shown that the equivalence
    \ref{item:mult-open-in-free-profinite-4}$\Leftrightarrow$\ref{item:mult-open-in-free-profinite-1} holds. Since $C$ is an arbitrary finite retract of $B$, applying that equivalence to $B$ instead of $A$, we
    get that the multiplication in $\Om BV$ is open.
    
    \ref{item:mult-open-in-free-profinite-5}$\Rightarrow$\ref{item:mult-open-in-free-profinite-4}.
  This implication is trivial.
\end{proof}

For pseudovarieties of semigroupoids of the form $\pv{\glob V}$, we also have the following sufficient condition for the multiplication in the semigroupoid $\Om A{\glob V}$ to be open.

\begin{proposition}
  \label{p:open-in-EA-implies-open-in-global}
  Let $A$ be a finite-vertex graph and \pv V be a pseudovariety of
  semigroups containing~$B_2$. If the multiplication in $\Om
  {E(A)}V$ is an open mapping, then the multiplication
  in $\Om A{\glob V}$ is also an open mapping.
\end{proposition}

\begin{proof}
  Let $m_A$ and $m_{E(A)}$ be the multiplication in $\Om A{\glob V}$
  and in $\Om {E(A)}V$, respectively. Assume that $m_{E(A)}$ is an
  open mapping and let $U$ be an open subset of $D(\Om A{\glob V})$.
  We wish to show that $m_A(U)$ is open.
  
  The function $\gamma\times \gamma\from\Om A{\glob V}\times \Om
  A{\glob V}\to \Om {E(A)}V\times \Om {E(A)}V$, which maps each
  $(x,y)\in \Om A{\glob V}\times \Om A{\glob V}$ to
  $(\gamma(x),\gamma(y))$, is an open mapping by
  Proposition~\ref{p:faithful-open-case-B2}. In particular, and since
  $D(\Om A{\glob V})$ is open in $\Om A{\glob V}\times \Om A{\glob
    V}$, the set $(\gamma\times \gamma)(U)$ is open in $\Om
  {E(A)}V\times \Om {E(A)}V$. By the assumption that $m_{E(A)}$ is an
  open mapping, and by commutativity of
  Diagram~\ref{eq:open-in-EA-implies-open-in-global-1}, it then
  follows that $\gamma(m_A(U))$ is open in $\Om {E(A)}V$.
    \begin{equation}
  \begin{split}
    \label{eq:open-in-EA-implies-open-in-global-1}
    \xymatrix@C=2.5 cm@R=1.0 cm{
      D(\Om A{\glob V})\ar[d]_{\gamma\times \gamma|}\ar[r]^{m_A}&\Om A{\glob V}\ar[d]^{\gamma}\\
      \Om {E(A)}V\times \Om {E(A)}V\ar[r]^{m_{E(A)}}&\Om {E(A)}V}
  \end{split}                                                    
  \end{equation}
  Proposition~\ref{p:faithful-open-case-B2} yields
  that $\gamma$ is a homeomorphism onto its image.
  Therefore, the set $m_A(U)$ is open in $\Om A{\glob V}$.
\end{proof}

The pseudovariety $\pv{\glob N}$ itself fails the equivalent conditions in Theorem~\ref{t:mult-open-in-free-profinite}. For instance, if $A$ is a finite alphabet with at least two letters then, for every $a\in A$, the two languages $\{a\}$ and $A^+$ are \pv N-recognizable, but their concatenation is not.

Let $\pv{\loc Sl}$ be the pseudovariety consisting of all finite
semigroups $S$ such that, for every idempotent $e$ of $S$, the
semigroup $eSe$ belongs to $\pv {Sl}$.

\begin{example}
  Let $B$ be the graph with only two vertices $v,w$ and two edges
  $x,y\in B(v,w)$, for which the semigroupoid $\Om B{\glob V}$ is
  finite, whence it has open multiplication, for any choice of
  pseudovariety of semigroups $\pv V$. Thus, to show that the converse
  of Proposition~\ref{p:open-in-EA-implies-open-in-global} fails, it
  suffices to choose a pseudovariety $\pv V$ containing $B_2$ such
  that the multiplication in $\Om{E(B)}{\glob V}=\Om{E(B)}V$ is not
  open. For this purpose, we take $\pv V=\pv{LSl}$. It amounts to
  straightforward syntactic calculations to show that the language
  $x^+y^+$ is $\pv{LSl}$-recognizable but $x^+y^+x^+y^+$ is not
  (cf.~\cite{Pin:1986;bk}). Hence, the set of $\pv{LSl}$-recognizable
  languages of $\bigl(E(B)\bigr)^+$ is not closed under concatenation
  and, therefore, multiplication in $\Om{E(B)}{LSl}$ is not open by
  the proof of \cite[Lemma~2.3]{Almeida&ACosta:2007a}.
\end{example}

We say that a pseudovariety of semigroups \pv V is
\emph{concatenation-closed} if for every finite alphabet~$A$, the set of $\pv V$-recognizable languages over $A$ is closed under concatenation (which means that the associated variety of languages
according to Eilenberg's correspondence \cite{Eilenberg:1976,Pin:1986;bk} is closed under concatenation). We adopt the same definition for pseudovarieties of finite semigroupoids, where finite alphabets are replaced with finite graphs. 

For the semigroup case, the following characterization of concatenation-closed pseudovarieties can be found in a paper of Chaubard, Pin and Straubing~\cite{Chaubard&Pin&Straubing:2006}, as a special case of a
  far more general result about the so-called \emph{varieties of stamps}. 
  This equivalence is closely related to earlier work of Straubing, who showed the validity of its natural counterpart for
  pseudovarieties of monoids~\cite{Straubing:1979a}. The pseudovariety of all finite aperiodic semigroups is denoted $\pv A$. We refer to~\cite{Rhodes&Steinberg:2009qt}
for the definition of the Mal'cev product $\pv V\malcev \pv W$ of pseudovarieties of semigroups $\pv V,\pv W$.
Bear in mind that $\pv V$ and $\pv W$ are both contained in $\pv V\malcev \pv W$. 

\begin{theorem}
 \label{t:concatenation-malcev}
  Let \pv V be a pseudovariety of semigroups. Then $\pv V$ is concatenation-closed if and only if the equality $\pv A\malcev\pv V=\pv V$ holds.
\end{theorem}

 \begin{remark}
    \label{r:concatenation-closed-gV-versus-V}
    If $\pv V$ is a concatenation-closed pseudovariety of semigroupoids,
    then the pseudovariety of semigroups $\pv W=\pv V\cap\pv S$ is concatenation-closed (cf.~Remark~\ref{r:recognizable-languages-trace}).
    Therefore, if $\pv V$ is a concatenation-closed pseudovariety of semigroupoids, then we have
    $\pv A\malcev\pv W=\pv W$ by Theorem~\ref{t:concatenation-malcev}, whence $\pv A\subseteq\pv V$, and in particular $\pv{N}\subseteq \pv V$.
 \end{remark}

 We can now deduce easily from Theorem~\ref{t:mult-open-in-free-profinite} the following characterization of concatenation-closed pseudovarieties of semigroupoids.

\begin{theorem}
  \label{t:sgpd-concatenation-closed-pseudovarieties}
  Let \pv V be a pseudovariety of semigroupoids. Then, the
  following conditions are equivalent:
  \begin{enumerate}
  \item \label{item:sgpd-concatenation-closed-pseudovarieties-1}
    $\pv V$ is concatenation-closed;
  \item \label{item:sgpd-concatenation-closed-pseudovarieties-2}
    for every finite-vertex graph $A$, the set of \pv V-recognizable languages over $A$ is closed under concatenation;
  \item \label{item:sgpd-concatenation-closed-pseudovarieties-3}
    $\pv V\supseteq \pv N$ and the multiplication in $\Om AV$ is open for every finite-vertex graph~$A$;
  \item \label{item:sgpd-concatenation-closed-pseudovarieties-4}
        $\pv V\supseteq \pv N$ and the multiplication in $\Om AV$ is open for every finite graph~$A$.
  \end{enumerate}
\end{theorem}

\begin{proof} The
  implication~\ref{item:sgpd-concatenation-closed-pseudovarieties-2}$\Rightarrow$\ref{item:sgpd-concatenation-closed-pseudovarieties-1}
  is trivial. If $\pv V$ is concatenation-closed, then $\pv V$
  contains~$\pv N$ by Remark~\ref{r:concatenation-closed-gV-versus-V},
  and so the
  implication~\ref{item:sgpd-concatenation-closed-pseudovarieties-1}$\Rightarrow$\ref{item:sgpd-concatenation-closed-pseudovarieties-3}
  follows immediately from
  Theorem~\ref{t:mult-open-in-free-profinite}. The
  implication~\ref{item:sgpd-concatenation-closed-pseudovarieties-3}$\Rightarrow$\ref{item:sgpd-concatenation-closed-pseudovarieties-4}
  is trivial. Finally, the
  implication~\ref{item:sgpd-concatenation-closed-pseudovarieties-4}$\Rightarrow$\ref{item:sgpd-concatenation-closed-pseudovarieties-2}
  also follows directly from
  Theorem~\ref{t:mult-open-in-free-profinite}.
\end{proof}

We next proceed to deduce the companion of Theorem~\ref{t:sgpd-concatenation-closed-pseudovarieties}
for pseudovarieties of semigroups.

\begin{theorem}
 \label{t:concatenation-closed-pseudovarieties}
  Let \pv V be a pseudovariety of semigroups. Then, the
  following conditions are equivalent:
  \begin{enumerate}
  \item \label{item:concatenation-closed-pseudovarieties-1}
    $\pv V$ is concatenation-closed;
  \item \label{item:concatenation-closed-pseudovarieties-2}
    for every alphabet $A$, if $K,L\subseteq A^+$ are \pv V-recognizable, then so is $KL\subseteq A^+$;
  \item \label{item:concatenation-closed-pseudovarieties-3}
    $\pv V\supseteq \pv N$ and the multiplication in $\Om AV$ is open for every alphabet $A$;
  \item \label{item:concatenation-closed-pseudovarieties-4}
     $\pv V\supseteq \pv N$ and the multiplication in $\Om AV$ is open for every finite alphabet~$A$;
  \end{enumerate}
\end{theorem}

\begin{proof}
  The implication \ref{item:concatenation-closed-pseudovarieties-2}$\Rightarrow$\ref{item:concatenation-closed-pseudovarieties-1} is trivial.

  To show the implication
  \ref{item:concatenation-closed-pseudovarieties-1}$\Rightarrow$\ref{item:concatenation-closed-pseudovarieties-4},
  observe that, assuming that
  $\ref{item:concatenation-closed-pseudovarieties-1}$ holds, then in
  view of Theorem~\ref{t:concatenation-malcev},
  we know that $\pv V$ contains $\pv A$, and in particular it contains
  $\pv N$. On the other hand, note that if $A$ is an alphabet and $\pv
  V$ is a pseudovariety of semigroups, then we have $\Om AV=\Om
  A{\glob V}$. Hence,
  \ref{item:concatenation-closed-pseudovarieties-4} follows from the
  implication $\ref{item:mult-open-in-free-profinite-1}
  \Rightarrow\ref{item:mult-open-in-free-profinite-2}$ in
  Theorem~\ref{t:mult-open-in-free-profinite}.

  The chain \ref{item:concatenation-closed-pseudovarieties-4}$\Rightarrow$\ref{item:concatenation-closed-pseudovarieties-3}$\Rightarrow$\ref{item:concatenation-closed-pseudovarieties-2} also follows clearly from Theorem~\ref{t:mult-open-in-free-profinite}.
\end{proof}

\begin{corollary}
  \label{c:sgpd-concatenation-closed-pseudovarieties}
  Let \pv V be a pseudovariety of semigroups. Then $\pv V$ is concatenation-closed if and only if $\glob\pv V$ is concatenation-closed.  
\end{corollary}

\begin{proof}
  Since $\pv S\cap \glob\pv V =\pv V$, the ``if'' part of the
  corollary is immediate by
  Remark~\ref{r:concatenation-closed-gV-versus-V}. Conversely, suppose
  that $\pv V$ is concatenation-closed. Then we have $\pv V=\pv
  A\malcev\pv V$, by
  Theorem~\ref{t:concatenation-closed-pseudovarieties}, thus $\pv V$
  contains $\pv A$. In particular, we have $B_2\in\pv V$. Let $A$ be a
  finite-vertex graph. Also by
  Theorem~\ref{t:concatenation-closed-pseudovarieties}, the
  multiplication in $\Om {E(A)}V$ is open. It then follows from
  Proposition~\ref{p:open-in-EA-implies-open-in-global} that the
  multiplication in $\Om A{\glob V}$ is open. Since $A$ is an
  arbitrary finite-vertex graph,
  Theorem~\ref{t:sgpd-concatenation-closed-pseudovarieties} then
  implies that $\glob\pv V$ is concatenation-closed.
\end{proof}

The preceding corollary suggests the following problem, left open.

\begin{problem}
  \label{prob:characterize-concatenation-closed}
  If $\pv V$ is a concatenation-closed pseudovariety of semigroups,
  which pseudovarieties of semigroupoids in the interval
  $(\pv{\glob V},\ell\pv V]$ are concatenation-closed?
\end{problem}

While this problem is vacuous for local concatenation-closed
pseudovarieties, there are many non-local pseudovarieties of interest,
such as the \emph{complexity pseudovarieties} $\pv {C}_n$ for $n\ge1$,
cf.~\cite[Section 4.16]{Rhodes&Steinberg:2009qt}
or~\cite{Rhodes&Steinberg:2006}.

We highlight the following result which will be used in subsequent proofs.

\begin{corollary}
  \label{c:open-mult-free-profinite}
  For every alphabet (respectively, finite-vertex graph) $A$ and
  every concatenation-closed pseudovariety of semigroups (respectively, semigroupoids) \pv V, every convergent net in~$\Om AV$ is
  fully factorizable.
\end{corollary}

\begin{proof}
  Combine Corollary~\ref{c:open-mult-fully-factorizable} with, respectively, Theorems~\ref{t:concatenation-closed-pseudovarieties} and~\ref{t:sgpd-concatenation-closed-pseudovarieties}.
\end{proof}

The proofs of both parts of the next proposition follow the same lines of reasoning, employed in
previous works, to establish the corresponding special cases pertaining to pseudovarieties
of semigroups and finite alphabets; see respectively~\cite[Proposition~2.4]{Almeida&ACosta:2007a} and~\cite[Lemma~3.8]{ACosta&Steinberg:2011}.
The second part is an application of Theorem~\ref{t:sgpd-concatenation-closed-pseudovarieties}.

\begin{proposition}
  \label{p:closure-of-factorial}
  Let $A$ be a finite-vertex graph.
  Consider a pseudovariety of semigroupoids $\pv V$ containing $\pv N$.
  Let $L\subseteq A^+$ be a factorial language.
  Suppose that at least one of the following conditions holds:
  \begin{enumerate}
  \item $L$ is $\pv V$-recognizable;\label{item:closure-of-factorial-1}
  \item \pv V is concatenation-closed.\label{item:closure-of-factorial-2}
  \end{enumerate}
  Then the closure $\clos{V}(L)$ is a factorial subset of $\Om AV$.
\end{proposition}

\begin{proof}
  Let $w\in\clos V(L)$. Take $u,v\in\Om AV$
  such that $w=uv$. We wish to show that $u,v\in\clos V(L)$.

  Case \ref{item:closure-of-factorial-1}. 
  As $L$ is $\pv V$-recognizable, by Theorem~\ref{t:V-recognizability}
  the closure $\clos V(L)$ is an open subset of $\Om AV$.
  As $A^+$ is dense in $\Om AV$, we may take a net $(u_i,v_i)_{i\in I}$ of elements of $D(A^+)$ converging in $\Om AV\times \Om AV$
  to $(u,v)$.
  Since $\lim u_iv_i=w$ and $\clos V(L)$ is open, there is
  $k\in I$ such that $i\geq k$ implies $u_iv_i\in \clos V(L)$.
  Take $i\geq k$. Since the elements of $A^+$ are isolated points of $\Om AV$, we have $u_iv_i\in L$.
  As $L$ is factorial, it follows that $u_i,v_i\in L$ for every $i\geq k$. This yields $u,v\in \clos V(L)$.

  Case \ref{item:closure-of-factorial-2}. Let $(w_i)_{i\in I}$ be a
  net in~$L$ converging to $w$. By
  Corollary~\ref{c:open-mult-free-profinite}, there is a net
  $(u_j,v_j)_{j\in J}$ in $\Om AV\times\Om AV$ converging to
  $(u,v)$, and a subnet $(w_{i_j})_{j\in J}$ of $(w_i)_{i\in I}$ such that
  $w_{i_j}=u_jv_j$ for each $j\in J$. Since $u_j$ and $v_j$ are
  factors in $\Om AV$ of~$w_{i_j}$, a member of~$A^+$, they
  both belong to $A^+$ by Corollary~\ref{c:A+-as-filter}. As $L$ is
  factorial in~$A^+$, it follows that $u_j,v_j\in L$ for every $j\in
  J$, thus $u,v\in~\clos V(L)$.
\end{proof}

\section{Equidivisible pseudovarieties of semigroupoids}
\label{sec:equid-pseud-semigr}

A semigroupoid  $S$ is said to be \emph{equidivisible}~\cite{McKnight&Storey:1969}
if whenever we have edges $u,v,x,y\in S$ such that $uv=xy$,
there is an edge $t\in S^I$ such that at least one of the following conditions holds:
\begin{itemize}
\item $ut=x$ and $v=ty$;
\item $xt=u$ and $y=tv$.
\end{itemize}

The following definition was introduced by the first two
authors~\cite{Almeida&ACosta:2017}: a pseudovariety of semigroups $\pv
V$ is \emph{equidivisible} when $\Om AV$ is equidivisible for every
finite alphabet~$A$. Every pseudovariety of semigroups contained in
the pseudovariety $\pv{CS}$ of all finite completely simple semigroups
is an equidivisible pseudovariety, as every completely simple
semigroup is equidivisible~\cite{McKnight&Storey:1969}. For other
pseudovarieties, several characterizations are given in~\cite[Theorem
6.2]{Almeida&ACosta:2023}, with some appearing in earlier work~\cite{Almeida&ACosta:2017}. 
We highlight some of these characterizations in the following
theorem. We recall that $\pv{\loc \pv I}$ is the pseudovariety of
all finite semigroups $S$ such that, for every idempotent $e$, one has
$eSe=e$. Note that $\pv{\loc \pv I}$ contains $\pv N$.

\begin{theorem}
  \label{t:equidivisible-pseudovarieties}
  The following conditions are equivalent for a pseudovariety of semigroups~\pv V not contained in~\pv {CS}:
  \begin{enumerate}
  \item \label{item:equidivisible-pseudovarieties-1}
    $\pv V$ is equidivisible;
  \item\label{item:equidivisible-pseudovarieties-2} for every alphabet $A$, the semigroup \Om AV is equidivisible;
  \item\label{item:equidivisible-pseudovarieties-3} the equality $\pv{LI\malcev V}=\pv V$ holds.
 \end{enumerate}
 Moreover, the pseudovariety $\pv{\loc \pv I}$ is the least equidivisible pseudovariety of semigroups not contained in $\pv{CS}$.
\end{theorem}

By Theorem~\ref{t:concatenation-closed-pseudovarieties}, a pseudovariety of semigroups $\pv V$ is concatenation-closed if and only if
$\pv{A\malcev V}=\pv V$. Since $\pv {LI}\subseteq \pv A$, Theorem~\ref{t:equidivisible-pseudovarieties} entails the following corollary.

\begin{corollary}
  \label{c:concatenation-closed-implies-equidivisible}
  Every concatenation-closed pseudovariety of semigroups is equidivisible.
\end{corollary}

Analogously, let us say that a pseudovariety of semigroupoids $\pv V$ is \emph{equidivisible}
if the semigroupoid $\Om AV$ is equidivisible for every finite graph $A$. 
The next proposition shows that the assumption that the graph is finite in this definition can be weakened.

\begin{proposition}
  \label{p:equidivisible-pseudovarieties-of-semigroupoids}
  A pseudovariety of semigroupoids $\pv V$ is equidivisible if and only if,
  for every finite-vertex graph $A$, the semigroupoid $\Om AV$ is equidivisible.
\end{proposition}

To establish this proposition
we use the following lemma, whose proof relies on a routine compactness
argument, given here for the reader's convenience.

\begin{lemma}
  \label{p:inverse-limit-of-equidivisible-semigroupoids}
  If the profinite semigroupoid $S$
  is a quotient inverse limit of equidivisible profinite semigroupoids, then $S$ is equidivisible.
\end{lemma}

\begin{proof}
  Take $S=\varprojlim_{i\in I}S_i$
  with underlying quotient inverse system
  \begin{equation*}
   \{\pi_{j,i}\from S_j\to S_i\mid i\leq j;\,i,j\in I\} 
  \end{equation*}
  of continuous semigroupoid homomorphisms
  such that $S_i$ is an equidivisible profinite semigroupoid for every $i\in I$.
  We denote by $\pi_i$ the induced projection $S\to S_i$.

  Let $u,v,x,y\in E(S)$ be such that $uv=xy$. As $S_i$ is equidivisible for every $i\in I$, the set $I$ is the union of the following two sets:
    \begin{align*}
    I_1&=\{i\in I\mid \exists t\in S_i^I(\alpha(\pi_i(x)),\alpha(\pi_i(u))):\pi_i(u)t=\pi_i(x)\text{ and }\pi_i(v)=t\pi_i(y)\},\\
   I_2&=\{i\in I\mid \exists t\in S_i^I(\alpha(\pi_i(y)),\alpha(\pi_i(v))):\pi_i(u)=t\pi_i(x)\text{ and }\pi_i(v)t=\pi_i(y)\}.
  \end{align*}
  At least one these sets is cofinal in $I$. Hence, without loss of generality, we may assume that $I=I_1$ up to taking a subnet.
  Then, for every $i\in I$, since $\pi_i$ is a quotient homomorphism (cf.~\cite[Section 3.2]{Engelking:1989}), the following set is nonempty:
  \begin{equation*}
    T_i=\{\tau\in S^I(\alpha(x),\alpha(u)):\pi_i(u\tau)=\pi_i(x)\text{ and }\pi_i(v)=\pi_i(\tau y)\}.
  \end{equation*}
  Let $T=\bigcap_{i\in I}T_i$.
  Note that $\tau\in T$ if and only if $u\tau=x$ and $v=\tau y$.
  Therefore, to establish the lemma it suffices to show that $T$ is nonempty.
  
  For every $i,j\in I$ such that $i\leq j$,
  it follows from the equality $\pi_{j,i}\circ\pi_j=\pi_i$
  that $T_j\subseteq T_i$.
  Let $F$ be a finite subset of $I$. Since $I$ is directed, there is $k\in I$
  such that $i\leq k$ for every $i\in F$. Hence $T_k$
  is a nonempty subset of  $\bigcap_{i\in F}T_i$. Since $(T_i)_{i\in I}$
  is a family of closed subsets of the compact space $S^I$, we deduce
  that the intersection $\bigcap_{i\in I}T_i$ is nonempty. As remarked, this shows that $S$ is equidivisible.
\end{proof}

 \begin{proof}[{Proof of Proposition~\ref{p:equidivisible-pseudovarieties-of-semigroupoids}}]
   By Corollary~\ref{c:inverse-limit-finite-retracts},
   the profinite semigroupoid $\Om AV$ is a quotient inverse limit $\varprojlim\Om {A_i}V$ such that $A_i$
   is finite for every $i\in I$. Since, by hypothesis, $\Om {A_i}V$ is equidivisible for every $i\in I$,
   it follows from Lemma~\ref{p:inverse-limit-of-equidivisible-semigroupoids}
   that $\Om AV$ is equidivisible.
 \end{proof}

\begin{remark}
  \label{r:equidivisible-V-vs-V-cap-S}
  Note that if the pseudovariety of semigroups $\pv V$ is such that $\pv{\glob V}$ is
  equidivisible, then $\pv{V}$ is also equidivisible in view of the equality $\Om A{\glob
    V}=\Om AV$ for every alphabet $A$ (see Remark~\ref{r:om-global-V-alphabet}).
\end{remark}

 The following proposition is a simple application of
 Proposition~\ref{p:faithful-open-case-B2}.

\begin{proposition}
  \label{p:equidivible-V-versus-global-V}
  Let $\pv V$ be a pseudovariety of semigroups containing $B_2$. Then $\pv V$ is  equidivisible
  if and only if $\pv{\glob V}$ is equidivisible.
\end{proposition}

\begin{proof}
  The ``if'' part of the proposition is given by
  Remark~\ref{r:equidivisible-V-vs-V-cap-S}.
  
  Conversely, suppose that $\pv V$ is equidivisible.  Let $u,v,x,y$ be pseudopaths of $\Om A{\glob V}$ such that $uv=xy$.
  Since $\gamma(uv)=\gamma(xy)$ and $\Om {E(A)}V$ is equidivisible,
  there is $t\in \Om {E(A)}V$ such that at least one of the following conjunctions holds:
  $\gamma(u)\cdot t=\gamma(x)$ and $\gamma(v)=t\cdot \gamma(y)$;
  or $\gamma(x)\cdot t=\gamma(u)$ and $\gamma(y)=t\cdot \gamma(v)$.
  Without loss of generality, we assume the former. 
  Since we are assuming that $\pv V$ contains $B_2$, it follows from Proposition~\ref{p:faithful-open-case-B2} that
  the image of $\gamma$ is a factorial subset of $\Om {E(A)}{\glob V}$,
  and so we have $t=\gamma(s)$ for some pseudopath $s$.
  Moreover, also by Proposition~\ref{p:faithful-open-case-B2},
  the equalities $\gamma(x)\cdot \gamma(s)=\gamma(u)$
  and $\gamma(y)=\gamma(s)\cdot \gamma(v)$
  yield the factorizations $xs=u$ and $y=sv$ in $\Om AV$. Hence $\Om AV$ is indeed equidivisible for every every finite-vertex graph~$A$.
\end{proof}

At this point, it is natural to ask the following question, which we leave as an open problem.

\begin{problem}
  \label{prob:characterize-equidivisibible}
  If\/ $\pv V$ is an equidivisible pseudovariety of semigroups, which
  semigroupoid pseudovarieties in the interval $(\pv{\glob V},\ell\pv V]$
  are equidivisible?
\end{problem}

 We close this section by proving the semigroupoid counterpart of Corollary~\ref{c:concatenation-closed-implies-equidivisible}. This also establishes a connection between Problems~\ref{prob:characterize-equidivisibible} and~\ref{prob:characterize-concatenation-closed}. 
\begin{corollary}
  \label{c:concatenation-closed-implies-equidivisible-semigroupoids}
  Every concatenation-closed pseudovariety of semigroupoids is equidivisible.
\end{corollary}

In the proof of this corollary, we use the following lemma.
 
\begin{lemma}
  \label{l:equidivisible-from-subsemigroupoid-to-semigroupoid}
  Let $S$ be a compact semigroupoid with open multiplication, and let~$T$ be a subsemigroupoid of $S$
  such that $T$ is factorial and dense in $S$. Then $S$ is equidivisible if and only if $T$ is equidivisible.
\end{lemma}

\begin{proof}
  It is clear that every factorial subsemigroupoid of an equidivisible semigroupoid is itself equidivisible.

  Let $S$ and $T$ be as in the statement, with $T$ equidivisible. Let
  $u,v,x,y$ be edges of $S$ such that $uv=xy$. Since $T$ is dense in
  $S$, there is a net $(u_i,v_i)_{i\in I}$ of elements of $T^2$
  converging in $S^2$ to $(u,v)$. As $S$ is fully factorizable
  (Corollary~\ref{c:open-mult-fully-factorizable}), there is a subnet
  $(u_{i_j},v_{i_j})_{j\in J}$ of $(u_i,v_i)_{i\in I}$ and a net $(x_j,y_j)_{j\in J}$ converging to $(x,y)$, such
  that the equality $u_{i_j}v_{i_j}=x_jy_j$ holds for
  every $j\in J$. Since $T$ is factorial, both $x_j$ and $y_j$ belong
  to $T$, for every $j\in J$. As $T$ is moreover equidivisible, the
  set $J$ is the union of the following two sets:
  \begin{align*}
    J_1&=\{j\in J\mid \exists t\in T:u_{i_j}t=x_j\text{ and }v_{i_j}=ty_j\},\\
   J_2&=\{j\in J\mid \exists t\in T:u_{i_j}=x_jt\text{ and }tv_{i_j}=y_j\}.
  \end{align*}
  At least one these sets is cofinal in $J$. Hence, without loss of generality, we may assume that $J=J_1$ up to taking a subnet. Then, for each $j\in J$,
  choose $t_j\in T$ such that $u_{i_j}t_j=x_j$ and $v_{i_j}=t_jy_j$. Let $t$ be a cluster point of the net
  $(t_j)_{j\in J}$. By continuity of the multiplication, we obtain $ut=x$ and $v=ty$. This shows that $S$ is equidivisible.
\end{proof}

\begin{proof}[Proof of Corollary~\ref{c:concatenation-closed-implies-equidivisible-semigroupoids}]
  Let $\pv V$ be a concatenation-closed pseudovariety of semigroupoids.
  Let $A$ be a finite-vertex graph. We know by Theorem~\ref{t:sgpd-concatenation-closed-pseudovarieties} that $\pv V$ contains $\pv N$
  and that the multiplication in $\Om AV$ is an open mapping.
  Recall that $A^+$ is dense in $\Om AV$ and (by Corollary~\ref{c:A+-as-filter}) factorial in $\Om AV$. Hence,
  as $A^+$ is an equidivisible semigroupoid, the result now follows immediately from Lemma~\ref{l:equidivisible-from-subsemigroupoid-to-semigroupoid}.
\end{proof}

\section{Prefix accessible pseudowords and pseudopaths}
\label{sec:recurrence}

In this section, we establish some properties of \emph{prefix accessible pseudopaths} (Definition~\ref{def:prefix-accessible-pseudopaths}), occasionally with the help of fully factorizable nets, via some of the main results established in Section~\ref{sec:open-multiplication}. In the preliminaries leading to that, we extend some known properties
about finite-length prefixes of pseudowords over finite alphabets to
pseudowords over arbitrary alphabets, and, more generally,
to pseudopaths over arbitrary finite-vertex pseudopaths.

We start by establishing conditions for the existence and uniqueness
of prefixes of any given finite length. In the following statement, we
refer to the pseudovariety~$\pv{K}_n$ of all finite semigroups in
which all products of $n$ factors are left zeroes.

\begin{proposition}
  \label{p:length-n-prefix}
  Let $A$ be a finite-vertex graph
  and fix an integer $n\geq 1$. Let $\pv{V}$ be a pseudovariety of
  semigroupoids containing $\pv{K}_n$ and $\pv N$. If $w\in\Om AV$ is a
  pseudopath of length at least $n$, possibly infinite, then there exists a unique
  pseudopath $u\in\Om AV$ of length $n$ and such that $u$ is a prefix of
  $w$.
\end{proposition}

\begin{proof}
  Let $(w_i)_{i\in I}$ be a net of paths in $A^+$ converging to $w$ in
  $\Om AV$. By continuity of the length homomorphism $\ell_A\from\Om
  AV\to\nn\cup\{\infty\}$, we may assume that $|w_i|\geq n$ for all
  $i\in I$. Let $u_i, v_i\in A^*$ be such that $w_i = u_iv_i$ and
  $|u_i|=n$. By taking a subnet, we may furthermore assume that the
  net $(u_i,v_i)_{i\in I}$ converges in $\Om AV\times (\Om AV)^I$ to
  some pair $(u,v)$. Note that $w=uv$ by continuity of multiplication,
  and that $|u|=n$ by continuity of $\ell_A$. Therefore, $u$ is a
  prefix of length $n$ of $w$.

  To prove uniqueness, suppose there is another factorization $w = st$
  where $|s|=n$, $s\in\Om AV$ and $t\in(\Om AV)^I$. By
  Corollary~\ref{c:separation-of-points}, if $u\neq s$ then there exists
  a finite graph $B$ and a continuous quotient homomorphism $\varphi\from
  \Om AV\to\Om BV$ such that $\varphi(A)=B$ and
  $\varphi(u)\ne\varphi(s)$.  Since $\varphi(A)\subseteq B$
  we must have $\ell_B(\varphi(a))=\ell_A(a)=1$
  for every edge $a$ of $A$.
  Therefore we have $\ell_B\circ\varphi=\ell_A$; in other words, $\varphi$ preserves length.
  In particular, both pseudopaths $\varphi(u)$ and $\varphi(s)$ have length $n$.
  As $B$ is finite, $\varphi(u)$ and $\varphi(s)$ are in fact words in $B^+$, of length $n$, by Proposition~\ref{p:finite-length-pseudopaths-in-finite-graph-case}.

  It is well known
  that for every finite alphabet $C$, the finite semigroup $\Om CK_n$ may be viewed as the set of all
  nonempty words of length at most $n$ over the alphabet~$C$,
  multiplication consisting in concatenating the words and taking the prefix of
  length $n$ if the result is longer than~$n$
  \cite[Section~5.2]{Almeida:1994a}. It follows that the natural
  projection $\pi\from\Om BV\to\Om {E(B)}K_n$ is injective on
  paths over $B$ of length at most~$n$. This leads to a contradiction: $\varphi(u)$ and $\varphi(s)$ are distinct paths over $B$ with length~$n$, while the above description of the
  operation in~$\Om {E(B)}K_n$ entails the equalities
  $\pi(\varphi(u))
    =\pi(\varphi(uv))
    =\pi(\varphi(st))
    =\pi(\varphi(s))$.
  Thus, the equality $u=s$ must hold.
\end{proof}

For a set $X$, let $w=w_0w_1w_2\cdots$ be an element of $X^\nn$, with
$w_n\in X$ for all $n\in\nn$. Given an interval $I$ of~$\nn$, we may
denote by $wI$ the (possibly infinite) sequence $(w_i)_{i\in I}$.
A~\emph{right-infinite path} over a graph $A$ is an element $w$ of
$\edg(A)^\nn$ such that $w{[}0,n{)}$ is a path over $A$, for every
$n\in\nn$. Proposition~\ref{p:length-n-prefix} leads to the following
definition, where $\pv K$ is the semigroup pseudovariety
$\bigcup_{n\geq 1}\pv K_n$. 
It follows from standard results in semigroup theory~\cite[Proposition~3.7.1]{Almeida:1994a} that $\pv K$ consists of all finite semigroups whose idempotents are left zeros, and in particular $\pv N\subseteq \pv K \subseteq \pv A$.

\begin{definition}
  \label{d:prefix-suffix}
  Let $w$ be a pseudopath of $\Om AV$, where $\pv{V}$ is a pseudovariety of semigroupoids containing
  $\pv K$.
  \begin{itemize}
  \item If $|w|\geq n$, where $n\in\nn$, we call the unique pseudoword $u$ of
    length~$n$ such that $w\in u(\Om AV)^I$ the \emph{prefix of length
      $n$ of $w$}, cf.~Proposition~\ref{p:length-n-prefix}.
  \item If $|w|=\infty$, we denote by $\overrightarrow{w}$ the
    right-infinite path over the graph  $\clos V(A)$ whose finite-length
    prefixes are those of $w$, cf.~Remark~\ref{r:pseudo-letters}. (Note that if $A$ is finite then $\clos V(A) = A$.)
  \end{itemize}
\end{definition}

It goes without saying that there is a dual of
Definition~\ref{d:prefix-suffix} where prefixes are replaced by suffixes, and where $\pv K$ is replaced by
its ``dual'' pseudovariety, denoted $\pv D$, which
consists of all finite semigroups where idempotents are right
zeros. It is worth mentioning that $\pv{\loc \pv I}$ is the least
pseudovariety containing $\pv K$ and $\pv D$.

Let $\pv V$ be a pseudovariety of semigroupoids containing $\pv N$. We
say that $\pv V$ is \emph{left finitely cancelable} if the following
condition holds: for every finite-vertex graph $A$, if the pseudopaths
$x,y,u,v\in(\Om AV)^I$ are such that the equality $xu=yv$ holds in
$\Om AV$ and $|x|=|y|\in\nn$, then $x=y$ and $u=v$. We may then denote
$v$ by $x^{-1}w$, or by $w^{(k)}$ where $w=xu$ and $k=|x|$. The dual
notion of \emph{right finitely cancelability} is defined in an analogous way. Finally, say that $\pv V$ is \emph{finitely cancelable} if it is
both right and left finitely cancelable.
  
\begin{proposition}
  \label{p:letter-super-cancelativity-semigroupoids}
  Every equidivisible pseudovariety of semigroupoids containing $\pv N$ is finitely cancelable.
\end{proposition}

\begin{proof}
  Let $\pv V$ be an equidivisible pseudovariety of semigroupoids containing $\pv N$ and let~$A$ be a finite-vertex graph.
  Take pseudopaths $x,y,u,v\in(\Om AV)^I$ such that $xu=yv$ and $|x|=|y|\in\nn$.
  By equidivisibilty, the equality $xu=yv$
  entails the existence of $t\in(\Om AV)^I$ such that at least one of the following conjunctions holds:
  $xt=y$ and $u=tv$; or $x=ty$ and $ut=v$.
  We then get $|x|+|t|=|y|$ or $|x|=|t|+|y|$. As we are assuming $|x|=|y|$,
  it follows that $t$ is an empty path, whence $x=y$ and $u=v$. This establishes that $\pv V$ is left finitely cancelable. With similar arguments,
  we deduce that $\pv V$ is right finitely cancelable.
\end{proof}

Replacing semigroupoids by semigroups and finite-vertex graphs by alphabets, we obtain corresponding definitions of left and right finitely cancelable pseudovariety of semigroups, and thus of
finitely cancelable pseudovariety of semigroups. The arguments used in Proposition~\ref{p:letter-super-cancelativity-semigroupoids} also establish the following proposition.

\begin{proposition}
  \label{p:letter-super-cancelativity-semigroups}
  Every equidivisible pseudovariety of semigroups containing $\pv N$ is finitely cancelable.
\end{proposition}

\begin{remark}
  Finitely cancelable pseudovarieties of semigroups are defined in~\cite{Almeida&ACosta&Costa&Zeitoun:2019} almost as here, the essential difference being that only finite alphabets
  are considered there. It is an easy exercise (not necessary for the remaining of the paper) to show that
  the two definitions are equivalent for all pseudovarieties of semigroups containing $\pv N$, using  the isomorphism
  $\Om AV\to \varprojlim_{\theta\in\Theta}\Om {A/{\theta}}V$
  from Corollary~\ref{c:inverse-limit-finite-retracts}
  and the fact that the canonical projections $\Om AV\to \Om {A/{\theta}}V$ preserve length. 
  Hence we can say that Proposition~\ref{p:letter-super-cancelativity-semigroups} of this paper coincides
  with Proposition 6.4 from paper~\cite{Almeida&ACosta&Costa&Zeitoun:2019}, where only finite alphabets are considered.
  See also Proposition 6.2 from~\cite{Almeida&ACosta&Costa&Zeitoun:2019} for a complete characterization of right finitely cancelable pseudovarieties of semigroups.
\end{remark}

\begin{definition}[Prefix accessible pseudopaths]
  \label{def:prefix-accessible-pseudopaths}
  Consider a finite-vertex graph $A$ and a pseudovariety of semigroupoids $\pv V$ containing $\pv N$.
  Let  $\dex V(w)$ denote the set of all cluster points in $\Om AV$ of the sequence $(w[0,n])_{n\in\nn}$.
  An element of a set of the form $\dex V(w)$ for some right-infinite path $w$ is said to be a \emph{prefix accessible} pseudopath.    
\end{definition}

Note that a prefix accessible pseudopath has infinite length. In fact, if $w$ is a right-infinite path over the finite-vertex graph $A$, then the set $\dex{V}(w)$ is given by the equality
\begin{equation*}
  \dex{V}(w) = \{ x\in\clos V(\{ w[0,n] \from n\in\nn\}) : |x| = \infty\}.
\end{equation*}  

\begin{remark}
  \label{r:projecting-prefix-accessible}
  For all pseudovarieties of semigroupoids $\pv{U}$ and $\pv{V}$
  such that $\pv{N}\subseteq\pv{V}\subseteq\pv{U}$,
  the fact that the natural projection $p_{\pv{U},\pv{V}}\from\Om AU\to\Om AV$ restricts to the identity on $A^+$
  entails the equality $\dex V(w) = p_{\pv{U},\pv{V}}(\dex U(w))$.
\end{remark}

The next proposition shows that the set $\dex V(w)$ is an $\green{R}$-class under mild conditions on $\pv V$.
Its proof makes use of Corollary~\ref{c:open-mult-free-profinite}.

\begin{proposition}
  \label{p:prefixes-R-class}
  Let  $w$ be a right-infinite path over the finite-vertex graph $A$.
  Consider a pseudovariety of semigroupoids $\pv V$ containing $\pv N$.
  Then the following properties hold:
  \begin{enumerate}
  \item The set $\dex{V}(w)$ is contained in an \green{R}-class of $\Om AV$.\label{item:prefixes-R-class-1}
  \item If\/ \pv V contains \pv K, then the \green{R}-class of $\Om AV$ containing $\dex{V}(w)$ is the
    set of infinite-length prefixes of elements of
    $\dex{V}(w)$.\label{item:prefixes-R-class-2}
  \item If\/ $\pv V$ is concatenation-closed, then $\dex{V}(w)$ is an \green{R}-class of $\Om AV$.\label{item:prefixes-R-class-3}
  \end{enumerate}
\end{proposition}

\begin{proof}
  \ref{item:prefixes-R-class-1}. Let $u,v\in \dex{V}(w)$. Then, there is a net $(u_i,v_i)_{i\in I}$
  of pairs of finite-length prefixes of $w$ converging in $\Om AV\times \Om
  AV$ to $(u,v)$. Fix $i\in I$. Since all finite prefixes of $w$ are
  $\le_{\green{R}}$-comparable elements of $A^+$, and since $u,v$ are infinite-length
  pseudowords, there is $k_i\in I$ such that $k_i\le j$ implies
  $u_j\le_{\green{R}}v_i$. Because $\le_{\green{R}}$ is a closed
  relation in $(\Om AV)^I$, we obtain $u\le_{\green{R}}v_i$ for every $i\in I$, and
  consequently $u\le_{\green{R}}v$. By symmetry, we also have
  $v\le_{\green{R}}u$, thus showing that $\dex V(w)$ is contained in
  an \green{R}-class of $\Om AV$.

  \ref{item:prefixes-R-class-2}.
  Assume that $\pv V$ contains $\pv K$, and denote by $R$ the $\green{R}$-class of $\Om AV$ containing
  $\dex  V(w)$, whose existence was established in the previous paragraph. 
  Let $p\in \Om AV$ be an infinite-length prefix of an element $v\in \dex  V(w)$.
  Then, as $\pv V$ contains $\pv K$, we have the equality $\overrightarrow{v}=\overrightarrow{p}$.
  From $v\in\dex V(w)$, we get, for every $n\in\nn$,  that $v\in w{[0,n]}\cdot \Om AV$,
  thus $p\leq_{\green{R}} w{[0,n]}$. Since the relation $\leq_{\green{R}}$ is closed in $\Om AV$, and
  $v$ is a cluster point of the sequence $(w[0,n])_{n\in\nn}$,
  we deduce that $p\leq_{\green{R}} v$. This shows that $p\in R$.
  To see that, conversely, every element of $R$ is
  an infinite-length prefix of some element of $\dex V(w)$,
  recall that the elements of $\dex V(w)$ have infinite length, and
  that the set of infinite-length pseudopaths of $\Om AV$ is an ideal (Remark~\ref{r:infinite-length-pseudowords-form-ideal}).

  \ref{item:prefixes-R-class-3}.
  Assume that $\pv V$ is closed under concatenation. 
  Note in particular that \pv V contains \pv A (see Remark~\ref{r:concatenation-closed-gV-versus-V}) and thus also \pv K. 
  Then, by the already established property~\ref{item:prefixes-R-class-2},
  to prove that the set $\dex{V}(w)$ is an $\green{R}$-class of $\Om AV$ it suffices to show
  that it is closed under taking infinite-length prefixes.
  Let $p\in \Om AV$ be an infinite-length prefix of an element $v\in \dex  V(w)$. Take $s\in(\Om AV)^I$ such that $v=ps$, and let $(v_i)_{i\in I}$
  be a net of finite prefixes of $w$ such that $v=\lim_{i\in I}v_i$.
  By Corollary~\ref{c:open-mult-free-profinite}, we may find a subnet $(v_{i_j})_{j\in J}$ of $(v_i)_{i\in I}$
  and a net $(p_j, s_j)_{j\in J}$ of composable pseudopaths
  converging to $(p,s)$ in $\Om AV\times (\Om AV)^I$ such that $p_j s_j = v_{i_j}$
  for every $j\in J$. Since, by assumption, $v_{i_j}\in A^+$, it follows from Corollary~\ref{c:A+-as-filter} that $p_j,s_j\in A^+$, for every $j\in J$.
  Then, the equality $v_{i_j}=p_j s_j$ also yields that $p_j$ is a finite prefix of $w$, for every $j\in J$.
  As $p$ has infinite length, it follows that $p\in \dex V(w)$, thus establishing that $\dex{V}(w)$ is closed under taking infinite-length prefixes. As already remarked, this concludes the proof.
\end{proof}

Without the assumption that $\pv V$ is concatenation-closed, it may happen that $\dex{V}(w)$ is not an \green{R}-class of $\Om AV$,
as seen in the next example. There, we use the fact that, for every profinite semigroup $S$ and $s\in S$, the sequence $(s^{n!})_{n\in\nn}$
converges in $S$ to an idempotent, denoted $s^\omega$ (see, for
instance, \cite[Proposition 3.9.2]{Almeida&ACosta&Kyriakoglou&Perrin:2020b}).

\begin{example}
  Consider the alphabet $A=\{\ltr{a},\ltr{b}\}$ and the equidivisible
  pseudovariety $\pv V=\pv{\loc \pv I}$. Let $w$ be the constant
  sequence $\ltr{a}\ltr{a}\ltr{a}\ltr{a}\ltr{a}\cdots \in A^\nn$. We
  clearly have $\ltr{a}^\omega\in \dex {\pv{\loc \pv I}}(w)$. On the
  other hand, in $\Om A{\loc \pv I}$ the equality
  $\ltr{a}^\omega=\ltr{a}^\omega \ltr{b} \ltr{a}^\omega$ holds, thus
  $\ltr{a}^\omega$ and $\ltr{a}^\omega \ltr{b}$ are
  $\green{R}$-equivalent elements of $\Om A{\loc \pv I}$. But
  $\ltr{a}^\omega \ltr{b}\notin \dex {\pv{\loc \pv I}}(w)$, since the
  suffix of length $1$ of any cluster point in $\Om A{\pv{\loc \pv
      I}}$ of the sequence $(\ltr{a}^n)_{n\geq 1}$ is clearly the
  letter $\ltr{a}$.
\end{example}

The following proposition is used in our parallel work~\cite{Almeida&ACosta&Goulet-Ouellet:2024b}.

\begin{proposition}\label{p:cutting-finite-prefixes-Pw}
  Let $A$ be a finite-vertex graph and $w$ be
  a right infinite path over~$A$. Consider a left finitely cancelable pseudovariety of semigroupoids $\pv V$ containing $\pv N$. For every $y\in \dex V(w)$ and $k\in\nn$, we have $y^{(k)}\in \dex V(w{[}k,\infty{)})$.
\end{proposition}

\begin{proof}
  Set $u=w{[}0,k{)}$. We may take a net $(y_i)_{i\in I}$ of finite prefixes
  of $w$ that converges to $y$ in $\Om A{V}$. Since $y$
  has infinite length, we may as well suppose that $u$ is a prefix of
  $y_i$, for every $i\in I$. For each $i\in I$, let $z_i=u^{-1}y_i$.
  Note that $z_i$ is a prefix of $w{[}k,\infty{)}$ and $\lim
  |z_i|=\infty$. By taking subnets, we may suppose that the net
  $(z_i)_{i\in I}$ converges to some pseudopath $z$ in $\Om A{V}$. Therefore, $z$
  belongs to $\dex{V}(w{[}k,\infty{)})$. Finally, we have $uz=\lim
  uz_i=\lim y_i=y$, thus $z=u^{-1}y=y^{(k)}$.
\end{proof}

For an alphabet $A$, an element $w$ of $A^\nn$
is said to be \emph{recurrent} if for every $n\in\nn$, there is some
$m>n$ such that $w[0,n)=w[m,m+n)$. Equivalently, $w\in A^\nn$ is
recurrent when every finite factor of $w$ occurs infinitely often in
$w$.

As an application of Proposition~\ref{p:prefixes-R-class}, we proceed
to deduce the following characterization of recurrent right-infinite paths, the main result of the present section. This characterization is necessary
for showing one the main theorems in~\cite{Almeida&ACosta&Goulet-Ouellet:2024b}.

\begin{theorem}
  \label{t:idempotents-in-Pw-pseudopath-version}
  Consider a finite-vertex graph $A$.
  Let $w$ be a right-infinite path over~$A$.
  The following statements are equivalent:
  \begin{enumerate}
  \item $w$ is recurrent;
    \label{i:idempotents-in-Pw-1}
  \item $\dex {Sd}(w)$ contains an idempotent;
    \label{i:idempotents-in-Pw-2}
  \item $\dex V(w)$ contains an idempotent for all pseudovarieties of semigroupoids $\pv V$ containing $\pv N$; \label{i:idempotents-in-Pw-3}
  \item $\dex{\glob\loc Sl}(w)$ contains an idempotent.
    \label{i:idempotents-in-Pw-4}
  \end{enumerate}
\end{theorem}

For the proof, we require the following lemma.

\begin{lemma}
  \label{l:LSl-clopen}
  Let $A$ be a finite-vertex graph, $\pv{V}$ be a pseudovariety of
  semigroupoids containing $\pv{\loc Sl}$ and $u,v\in A^+$. Then the set
  $uE\bigl((\Om AV)^I\bigr) vE\bigl((\Om AV)^I\bigr)$
  is clopen in $\Om AV$.
\end{lemma}

\begin{proof}
  It is well known that, if $D$ is a finite alphabet and $x,y\in D^+$,
  then the languages $xD^*$ and $D^*yD^*$ are $\pv{\loc Sl}$-recognizable
  in $D^+$ \cite[Theorem~5.2.1]{Pin:1986;bk}. By the equality
  \begin{equation*}
    uD^*vD^* = uD^*\cap\left(\bigcup_{w\in D^{|u|}}D^*wvD^*\right),
  \end{equation*}
  it follows that $uD^*vD^*$ is also $\pv{\loc Sl}$-recognizable in~$D^+$.
  
  By Theorem~\ref{t:V-recognizability}, since $\clos
  V\bigl(uE(A^*)vE(A^*)\bigr) = uE\bigl((\Om AV)^I\bigr)vE\bigl((\Om
  AV)^I\bigr)$, it suffices to show that the language $uE(A^*)vE(A^*)$
  is \pv V-recognizable in~$A^+$. Let $B$ be the alphabet whose
  letters are the edges of $A$ that are factors in $A^+$ of $u$ or
  $v$. Take an extra letter $c$ not in $B$, and consider the alphabet
  $C=B\cup\{c\}$. Let $\varphi\from A^+\to C^+$ be the homomorphism
  fixing the edges of $B$ and sending all other edges of $A$ to $c$.
  Since $\varphi^{-1}(uC^*vC^*) =uE(A^*)vE(A^*)$ it follows from the
  preceding paragraph that $uE(A^*)vE(A^*)$ is \pv V-recognizable
  in~$A^+$, which completes the proof of the lemma.
\end{proof}

\begin{proof}[{Proof of Theorem~\ref{t:idempotents-in-Pw-pseudopath-version}}]
  First, let us prove the implication
  \ref{i:idempotents-in-Pw-1}$\Rightarrow$\ref{i:idempotents-in-Pw-2}.
  Let $u\in \dex {Sd}(w)$ and take a net $(u_i)_{i\in I}$ of finite
  prefixes of $w$ converging to $u$. Assuming that $w$ is recurrent,
  for each $i\in I$ there is a path $v_i$ such that $u_iv_iu_i$ is a
  prefix of $w$. Taking subnets, we may assume that $(v_i)_{i\in I}$
  converges to some pseudopath $v\in\Om A{Sd}$. Note that $uvu = \lim
  u_iv_iu_i$ belongs to $\dex {Sd}(w)$. By
  Proposition~\ref{p:prefixes-R-class}\ref{item:prefixes-R-class-1}, there exists $t\in(\Om A{Sd})^I$ such
  that $u = uvut$.
  Observe that $uv$ and $t$ are loops at $\init(u)$.
  By induction, we obtain
  $u=(uv)^kut^k$ for every $k\in\nn$, thus $u = (uv)^\omega
  ut^\omega$.
  Since the pseudovariety $\pv{Sd}$ is concatenation-closed, it then
  follows from Proposition~\ref{p:prefixes-R-class}\ref{item:prefixes-R-class-3}
  that the idempotent $(uv)^\omega$ belongs to $\dex {Sd}(w)$.

  The implications \ref{i:idempotents-in-Pw-2}$\Rightarrow$\ref{i:idempotents-in-Pw-3}
  is immediate in view of~Remark~\ref{r:projecting-prefix-accessible}, while the implication \ref{i:idempotents-in-Pw-3}$\Rightarrow$\ref{i:idempotents-in-Pw-4}
  is trivial.

  It remains only to establish the implication
  \ref{i:idempotents-in-Pw-4}$\Rightarrow$\ref{i:idempotents-in-Pw-1}.
  Assume that $e$ is an idempotent in $\dex{\glob\loc Sl}(w)$. Let $u$
  be a finite prefix of $w$, so that $e = ut$ for some $t\in\Om
  A{\glob\loc Sl}$. Let $K = uE\bigl((\Om A{\glob\loc
    Sl})^I\bigr)uE\bigl((\Om A{\glob\loc Sl})^I\bigr)$, which is a
  clopen subset of $\Om A{\glob\loc Sl}$ by Lemma~\ref{l:LSl-clopen}.
  Since $e$ is idempotent, we have that $e=utut \in K$. As $e\in
  \dex{\glob\loc Sl}(w)$, it follows that there is a finite prefix of
  $w$ in $K$, which, by Corollary~\ref{c:A+-as-filter}, must be of the
  form $uxuy$ for some $x,y\in A^*$. Therefore, as $uxu$ is a prefix
  of $w$ with $x\in A^*$, we conclude that $w$ is recurrent.
\end{proof}

\begin{remark}\label{r:semigroupoids-vs-categories}
  In this paper, we opted to work primarly in the framework of pseudovarieties of semigroupoids instead of
  \emph{pseudovarieties of categories}, as the former encompasses more relevant situations.
  Concerning the study of profinite objects, the category framework
  is the preferred one in~\cite{Jones:1996}, while the semigroupoid one is preferred in~\cite{Almeida&Weil:1996}. 
  But as discussed in both papers, the two perspectives are closely related.
  In particular, denoting by $\pv{Cat}$ the pseudovariety of all finite categories, and recalling
  that $\pv{Sd}$ denotes the pseudovariety of all finite semigroupoids, the corresponding free objects are related by
  the equality $\Om A{Cat}=(\Om A{Sd})^I$, for every finite-vertex graph
  $A$ as it is easy to see that the profinite category $(\Om A{Sd})^I$
  is free over the graph $A$.
  Hence, it is reasonable to define $\dex {Cat}(w)$ as being $\dex {Sd}(w)$.
\end{remark}

In particular, Theorem~\ref{t:idempotents-in-Pw-pseudopath-version}
has the following immediate corollary. This is crucially used in our companion paper~\cite{Almeida&ACosta&Goulet-Ouellet:2024b},
where we opt to work with $\pv{Cat}$ instead of $\pv{Sd}$ for the sake of a somewhat more straightforward presentation.

\begin{corollary}
  \label{c:idempotents-in-Pw-pseudopath-version}
  Consider a finite-vertex graph $A$.
  Let $w$ be a right-infinite path over~$A$.
  Then $w$ is recurrent if and only if $\dex {Cat}(w)$ contains an idempotent. 
\end{corollary}

\section{Stabilizers}
\label{sec:stabilizers}

Consider a semigroupoid $S$. The \emph{right stabilizer} of an edge $x$ of $S$ is the set
\begin{equation*}
  \rstab(x)=\{y\in E(S^I)\from xy=x\}.
\end{equation*}
Note that $\rstab(x)$ is a submonoid of the local monoid $S^I(x)$. If moreover $S$ is a topological semigroupoid,
then $\rstab(x)$ is a closed submonoid of the topological monoid~$S^I(x)$.

Following the terminology from~\cite{Henckell&Rhodes&Steinberg:2010b},
when studying a subsemigroupoid $T$ of a semigroupoid~$S$, we
designate as an \emph{internal $\green{L}$-chain of $T$} a set $C$ of
elements of $T$ such that, for all $x,y\in C$, one has
$x\leq_{\green{L}(T)}y$ or $y\leq_{\green{L}(T)}x$. This terminology
serves the purpose of avoiding confusion between the relations
$\leq_{\green{L}(S)}$ and $\leq_{\green{L}(T)}$. If $C=T$ then we simply say that $T$ is an
internal $\green{L}$-chain. Likewise, we say that an element $x$ of
$T$ is \emph{internally regular} if it is regular in $T$, that is, if
$x\in xTx$.

We need the following fact.

\begin{lemma}
  \label{l:internal-l-chain}
  Let $S$ be a semigroupoid and $x$ an edge of $S$. If $S$ is equidivisible, then $\rstab(x)$ is an internal $\green{L}$-chain.
\end{lemma}

\begin{proof}
  Let $T=\rstab(x)$ and let $y,z\in T$. As $xy=xz$ and $S$ is equidivisible, there is $t\in S$ such that
  either $xt=x$ and $z=ty$, or $xt=x$ and $y=tz$. In particular, $t\in T$ and either $z\leq_{\green{L}(T)}y$ or $y\leq_{\green{L}(T)}z$.
\end{proof}

Our focus is on stabilizers in equidivisible relatively free profinite semigroups and semigroupoids.

\begin{theorem}
  \label{t:right-stabilizers-L-chain}
  Let $\pv V$ be an equidivisible pseudovariety of semigroups
  and $A$ be an arbitrary alphabet. Let $x$ be a pseudoword of $\Om AV$.
  Then the monoid $\rstab(x)$ is an internal $\green{L}$-chain and its internally regular elements are idempotents.
\end{theorem}

\begin{proof}
  By Theorem~\ref{t:equidivisible-pseudovarieties}, the semigroup $\Om AV$ is equidivisible.
  Hence, by Lemma~\ref{l:internal-l-chain}, the monoid $\rstab(x)$ is an internal $\green{L}$-chain.
  Let $y$ be an internally regular element of $\rstab(x)$. We wish to show that $y$ is idempotent. For the simple case where $\pv V$ is contained in $\pv{CS}$, see~\cite[Lemma 9.2]{Almeida&ACosta&Costa&Zeitoun:2019}.

  Assume that $\pv V$ is not contained in $\pv{CS}$. By Theorem~\ref{t:equidivisible-pseudovarieties}, we must have $\pv V=\pv{LI}\malcev\pv V$.
  By Proposition~\ref{p:letter-super-cancelativity-semigroups}, $\Om AV$ is finitely cancelable.
  In~\cite{Almeida&ACosta&Costa&Zeitoun:2019}, the definition of ``finitely cancelable'' is extended in a natural way
  to all finitely generated compact semigroups, subsuming the case of
  relatively free profinite semigroups over finite sets.
  It is shown in~\cite[Theorem 9.1]{Almeida&ACosta&Costa&Zeitoun:2019} that
  if $S$ is a finitely generated profinite equidivisible semigroup that is finitely cancelable, and $s\in S$, then every internally regular element of $\rstab(s)$
  must be idempotent. This settles the theorem in the case where $A$ is finite.

  Assume next that $A$ is an arbitrary alphabet, not necessarily finite. 
  Suppose that $y\neq y^2$.
  By Corollary~\ref{c:separation-of-points}
  there is a finite alphabet $B$ and a continuous onto homomorphism $\eta\from\Om AV\to \Om BV$
  such that $\eta(y)\neq\eta(y^2)$. On the other hand, we clearly have $\eta(\rstab(x))\subseteq \rstab(\eta(x))$.
  In particular, $\eta(y)$ is an internally regular element of $\rstab(\eta(x))$. From the established case of finite alphabets, it follows that $\eta(y)=\eta(y)^2$, a contradiction.
  To avoid the contradiction, we must have $y=y^2$.
\end{proof}

 \begin{corollary}
   \label{c:category-right-stabilizers-L-chain}
   Let $\pv V$ be an equidivisible pseudovariety of semigroups
   and $A$ be a finite-vertex graph. Let $x$ be a pseudopath of $\Om A{\glob V}$.
   Then the internally regular elements of $\rstab(x)$ are idempotent. If moreover~$\pv V$ contains $B_2$, then $\rstab(x)$
   is an internal $\green{L}$-chain.
 \end{corollary}

 \begin{proof}
   Consider the natural
   mapping $\gamma\from\Om A {\glob V}\to \Om {E(A)}V$.
   Let $y$ be an internally regular element of $\rstab(x)$.
   Then $\gamma(y)$ is an internally regular element of $\rstab(\gamma(x))$. It then follows from Theorem~\ref{t:right-stabilizers-L-chain}
   that $\gamma(y)=\gamma(y^2)$. As $\gamma$ is faithful (Proposition~\ref{p:faithfulness}), this establishes that $y$ is idempotent.
   If moreover $\pv V$ contains $B_2$, then by Proposition~\ref{p:equidivible-V-versus-global-V} the pseudovariety $\pv{\glob V}$ is equidivisible,
   and so $\rstab(x)$ is an internal $\green{L}$-chain by Lemma~\ref{l:internal-l-chain}.
 \end{proof}

 Next, we make use of the following classical notion: the \emph{kernel} of a semigroup $S$ is the minimal nonempty two-sided ideal, provided such an ideal exists. When it exists, the kernel consists of internally regular elements. It is well known that every compact semigroup has a kernel (cf.~\cite[Corollary 3.1.15]{Rhodes&Steinberg:2009qt}).
 \begin{corollary}
    \label{c:right-stabilizer-kernel}
    Let $A$ be a finite-vertex graph and\/ $\pv V$ be an
    equidivisible pseudovariety of semigroups. Let $x$ be a pseudopath
    of $\Om A{\glob V}$. Then the kernel of $\rstab(x)$ is a left-zero semigroup.
 \end{corollary} 

 \begin{proof}
   Since the kernel $K$ of $\rstab(x)$ exists and consists of internally regular elements, it follows from Corollary~\ref{c:category-right-stabilizers-L-chain} that $K$ is a band and an internal $\green{L}$-chain.
   Since the kernel of a profinite semigroup is completely simple, we conclude
   that $K$ is in fact an $\mathcal{R}$-trivial band.
 \end{proof}

\begin{remark}
  \label{r:stabilizers-rhodes-steinberg}
  Building on the work of Elston about expansions of semigroups~\cite{Elston:1999}, Rhodes and Steinberg also obtained sufficient conditions for a pseudovariety of semigroups $\pv V$ to satisfy the property that, for every finite alphabet $A$ and every $x\in\Om AV$,
  the semigroup $\rstab(x)$ is an $\green{L}$-chain of $\Om AV$ (a condition weaker than being an internal $\green{L}$-chain) whose internally regular elements are idempotent~\cite[Theorem 13.1]{Rhodes&Steinberg:2001}.
  These sufficient conditions hold in particular when $\pv V = \pv S$ or $\pv V = \pv A$~\cite[Corollary 13.2]{Rhodes&Steinberg:2001}.
  The approach of Rhodes and Steinberg is rather sophisticated and
  different from the one we used to deduce Theorem~\ref{t:right-stabilizers-L-chain}.
  In the same vein, they also gave sufficient conditions on a profinite monoid $M$, satisfied in particular for free profinite monoids, guaranteeing that for all $x\in M$, the stabilizer $\rstab(x)$ is an $\green{R}$-trivial band~\cite[Corollaries 14.4 and 14.5]{Rhodes&Steinberg:2001}.
\end{remark}

 The next theorem is the main result of the section. It plays a crucial role in our companion paper \cite{Almeida&ACosta&Goulet-Ouellet:2024b}. 

\begin{theorem}
  \label{t:net-characterization-of-L-minimal-right-stabilizers}
  Let $A$ be a finite-vertex graph and\/  $\pv V$ be a concatenation-closed pseudovariety of semigroups.
  Let $x$ be a prefix accessible pseudopath of $\Om A{\glob V}$. Then
  a pseudopath $y$ of $\Om A{\glob V}$ belongs to the kernel of\/ $\rstab(x)$
  if and only if there is a net $(x_i)_{i\in I}$ of finite-length prefixes of $x$ such that
  $x_i\to x$ and $x_i^{-1}x\to y$.
\end{theorem}

\begin{proof}
  In what follows, one should bear in mind that
  $\pv V\supseteq \pv A\supseteq \pv N$ (cf.~Theorem~\ref{t:concatenation-closed-pseudovarieties}),
  that $\pv{\glob V}$ is concatenation-closed (cf.~Corollary~\ref{c:sgpd-concatenation-closed-pseudovarieties}),
  and that $\pv V$ and $\pv{\glob V}$ are both equidivisible pseudovarieties
  (cf.~Corollaries~\ref{c:concatenation-closed-implies-equidivisible} and~\ref{c:concatenation-closed-implies-equidivisible-semigroupoids}).

  Take a pseudopath $y\in\Om A{\glob V}$ for which there is a net
  $(x_i)_{i\in I}$ of finite-length prefixes of $x$ such that $x_{i}\to x$
  and $x_i^{-1}x\to y$. Let $y_i=x_i^{-1}x$. Since $x=x_iy_i$, taking
  limits we get $x=xy$, thus $y\in\rstab(x)$. Let $z\in \rstab(x)$.
  Since
  \begin{equation*}
    x_iy_i=x=xz=x_iy_iz,
  \end{equation*}
  canceling (by Proposition~\ref{p:letter-super-cancelativity-semigroupoids}) the finite-length prefix $x_i$ we obtain $y_i=y_iz$. Taking limits,
  we see that $y=yz$, and so $y$ indeed belongs to the kernel of $\rstab(x)$.
  
  Conversely, suppose that $y$ belongs to the kernel of $\rstab(x)$.
  Since $x$ is prefix accessible, we know that $\overrightarrow{x}$ is
  a right-infinite path over $A$ and that there is some net
  $(x_i)_{i\in I}$ of finite-length prefixes of $x$ such that
  $x_{i}\to x$. Set $z_i=x_i^{-1}x$, for each~$i\in I$. 
  Up to taking a subnet, we may assume that the net $(z_i)_{i\in I}$ converges to an element $z\in\Om A{\glob V}$.
  By the first part of the proof, $z$ belongs to the kernel of $\rstab(x)$.
  As $\lim x_i=x=xy$, it follows from
  Corollary~\ref{c:open-mult-free-profinite} that there are nets
  $(x_j')_{j\in J}$ and $(y_j')_{j\in J}$ in $\Om A{\glob V}$ such
  that $\lim x_j'=x$, $\lim y_j'=y$ and a subnet
  $(x_{i_j})_{j\in J}$ of $(x_i)_{i\in I}$ satisfying
  $x_{i_j}=x_j'y_j'$ for every $j\in J$. Note that for every
  $j\in J$, since $x_{i_j}$ is a prefix of the right-infinite
  path $\overrightarrow{x}$ over $A$, it follows from
  Corollary~\ref{c:A+-as-filter} that $x_j'$ and $y_j'$ are paths over
  $A$, with $x_j'$ being moreover a prefix of $x$.

  Consider the pseudopath $y_j=y_j'z_{i_j}$. Then we have
  \begin{equation*}
    x=x_{i_j}z_{i_j}=x_j'y_j'z_{i_j}=x_j'y_j,
  \end{equation*}
  whence $y_j=(x_j')^{-1}x$. 
  Moreover, we have
  \begin{equation*}
    \lim_{j\in J} y_j = \lim_{j\in J} y'_j z_{i_j} = yz = y,
  \end{equation*}
  since $y$ and $z$ both belong to the kernel of $\rstab(x)$, which is a left-zero semigroup by Corollary~\ref{c:right-stabilizer-kernel}.
  Since $(x_j')_{j\in J}$ is a net of
  finite-length prefixes of $x$ converging to $x$, this concludes the proof of the theorem.
\end{proof}

In light of Remark~\ref{r:semigroupoids-vs-categories}, we immediately extract from
Corollary~\ref{c:right-stabilizer-kernel} and
Theorem~\ref{t:net-characterization-of-L-minimal-right-stabilizers} the following corollaries. These results are used in our related work~\cite{Almeida&ACosta&Goulet-Ouellet:2024b}.

\begin{corollary}
    \label{c:CAT-category-right-stabilizers-L-chain}
    Let $A$ be a finite-vertex graph. Let $x$ be a pseudopath of $\Om A{Cat}$. Then the kernel of $\rstab(x)$ is a left-zero semigroup.
\end{corollary}

\begin{corollary}
  \label{c:net-characterization-of-L-minimal-right-stabilizers}
  Let $A$ be a finite-vertex graph.
  Let $x$ be a prefix accessible pseudopath of~$\Om A{Cat}$. Then
  a pseudopath $y$ of $\Om A{Cat}$ belongs to the kernel of\/ $\rstab(x)$
  if and only if there is a net $(x_i)_{i\in I}$ of finite-length prefixes of $x$ such that
  $x_i\to x$ and $x_i^{-1}x\to y$.
\end{corollary}

\section*{Acknowledgments}

The first author acknowledges partial support by CMUP (Centro de
Matemática da Universidade do Porto), member of LASI (Intelligent
Systems Associate Laboratory), which is financed by Portuguese funds
through FCT (Fundação para a Ciência e a Tecnologia, I. P.) under the
project UIDB/00144.

The second author acknowledges financial support from the
Centre for Mathematics of the University of Coimbra - UID/00324.

The third author was supported by the Czech Technical University Global Postdoc Fellowship program.

\inputencoding{latin1}
\bibliographystyle{abbrv} 
\bibliography{../biblio/sgpabb,../biblio/ref-sgps}

\begin{thebibliography}{10}

\bibitem{Almeida:2025}
J.~Almeida.
\newblock Pseudovarieties of semigroups.
\newblock {\em Asian-Eur. J. Math.}, 0(0):2540007, 0.

\bibitem{Almeida:1994a}
J.~Almeida.
\newblock {\em Finite Semigroups and Universal Algebra}.
\newblock World Scientific, Singapore, 1994.
\newblock {E}nglish translation.

\bibitem{Almeida:1996c}
J.~Almeida.
\newblock A syntactical proof of locality of {DA}.
\newblock {\em Int. J. Algebra Comput.}, 6:165--177, 1996.

\bibitem{Almeida:2003cshort}
J.~Almeida.
\newblock Profinite semigroups and applications.
\newblock In V.~B. Kudryavtsev and I.~G. Rosenberg, editors, {\em Structural
  theory of automata, semigroups and universal algebra}, pages 1--45, New York,
  2005. Springer.

\bibitem{Almeida&ACosta:2007a}
J.~Almeida and A.~Costa.
\newblock Infinite-vertex free profinite semigroupoids and symbolic dynamics.
\newblock {\em J. Pure Appl. Algebra}, 213:605--631, 2009.

\bibitem{Almeida&ACosta:2017}
J.~Almeida and A.~Costa.
\newblock Equidivisible pseudovarieties of semigroups.
\newblock {\em Publ. Math. Debrecen}, 90:435--453, 2017.

\bibitem{Almeida&ACosta:2015hb}
J.~Almeida and A.~Costa.
\newblock {\em Handbook of {A}uto{M}ath{A}}, volume I. Theoretical Foundations,
  chapter 17. Profinite topologies, pages 615--652.
\newblock European Math. Soc. Publ. House, 2021.

\bibitem{Almeida&ACosta:2023}
J.~Almeida and A.~Costa.
\newblock Equidivisibility and profinite coproduct.
\newblock {\em Quaest. Math.}, 46(11):2243--2275, 2023.

\bibitem{Almeida&ACosta&Costa&Zeitoun:2019}
J.~Almeida, A.~Costa, J.~C. Costa, and M.~Zeitoun.
\newblock The linear nature of pseudowords.
\newblock {\em Publ. Mat.}, 63:361--422, 2019.

\bibitem{Almeida&ACosta&Goulet-Ouellet:2024b}
J.~Almeida, A.~Costa, and H.~Goulet-{O}uellet.
\newblock Profinite approach to {S}-adic shift spaces {I}: Saturating directive
  sequences.
\newblock In preparation.

\bibitem{Almeida&ACosta&Kyriakoglou&Perrin:2020b}
J.~Almeida, A.~Costa, R.~Kyriakoglou, and D.~Perrin.
\newblock {\em Profinite semigroups and symbolic dynamics}, volume 2274 of {\em
  Lect. Notes in Math.}
\newblock Springer, Cham, 2020.

\bibitem{Almeida&Steinberg:2000a}
J.~Almeida and B.~Steinberg.
\newblock On the decidability of iterated semidirect products and applications
  to complexity.
\newblock {\em Proc. London Math. Soc.}, 80:50--74, 2000.

\bibitem{Almeida&Steinberg:2008}
J.~Almeida and B.~Steinberg.
\newblock Rational codes and free profinite monoids.
\newblock {\em J. London Math. Soc. (2)}, 79:465--477, 2009.

\bibitem{Almeida&Weil:1996}
J.~Almeida and P.~Weil.
\newblock Profinite categories and semidirect products.
\newblock {\em J. Pure Appl. Algebra}, 123:1--50, 1998.

\bibitem{Banaschewski:1983}
B.~Banaschewski.
\newblock The {B}irkhoff theorem for varieties of finite algebras.
\newblock {\em Algebra Universalis}, 17:360--368, 1983.

\bibitem{Chaubard&Pin&Straubing:2006}
L.~Chaubard, J.-E. Pin, and H.~Straubing.
\newblock Actions, wreath products of \protect{$\mathcal C$}-varieties and
  concatenation product.
\newblock {\em Theor. Comp. Sci.}, 356:73--89, 2006.

\bibitem{Clifford&Preston:1961}
A.~H. Clifford and G.~B. Preston.
\newblock {\em The Algebraic Theory of Semigroups}, volume~I.
\newblock Amer. Math. Soc., Providence, R.I., 1961.

\bibitem{ACosta&Steinberg:2011}
A.~Costa and B.~Steinberg.
\newblock Profinite groups associated to sofic shifts are free.
\newblock {\em Proc. London Math. Soc.}, 102:341--369, 2011.

\bibitem{Eilenberg:1976}
S.~Eilenberg.
\newblock {\em Automata, Languages and Machines}, volume~B.
\newblock Academic Press, New York, 1976.

\bibitem{Elston:1999}
G.~Z. Elston.
\newblock Semigroup expansions using the derived category, kernel, and {M}alcev
  products.
\newblock {\em J. Pure Appl. Algebra}, 136(3):231--265, 1999.

\bibitem{Engelking:1989}
R.~Engelking.
\newblock {\em General Topology}.
\newblock Number~6 in Sigma Series in Pure Mathematics. Heldermann Verlag
  Berlin, 1989.
\newblock Revised and completed edition.

\bibitem{Grillet:1995bk}
P.-A. Grillet.
\newblock {\em Semigroups, an introduction to the structure theory}, volume 193
  of {\em Monographs and Textbooks in Pure and Applied Mathematics}.
\newblock Marcel Dekker, Inc., New York, 1995.

\bibitem{Hart&Nagata&Vaughan:2004egt}
K.~P. Hart, J.~Nagata, and J.~E. Vaughan, editors.
\newblock {\em Encyclopedia of {G}eneral {T}opology}.
\newblock Elsevier Science Publishers, B.V., Amsterdam, 2004.

\bibitem{Henckell&Rhodes&Steinberg:2010b}
K.~Henckell, J.~Rhodes, and B.~Steinberg.
\newblock A profinite approach to stable pairs.
\newblock {\em Int. J. Algebra Comput.}, 20:269--285, 2010.

\bibitem{Howie:1976}
J.~M. Howie.
\newblock {\em An Introduction to Semigroup Theory}.
\newblock Academic Press, London, 1976.

\bibitem{Jones:1996}
P.~R. Jones.
\newblock Profinite categories, implicit operations and pseudovarieties of
  categories.
\newblock {\em J. Pure Appl. Algebra}, 109:61--95, 1996.

\bibitem{Kelley:1975}
J.~L. Kelley.
\newblock {\em General topology}, volume No. 27 of {\em Graduate Texts in
  Mathematics}.
\newblock Springer-Verlag, New York-Berlin, 1975.
\newblock Reprint of the 1955 edition [Van Nostrand, Toronto, Ont.].

\bibitem{Lallement:1979}
G.~Lallement.
\newblock {\em Semigroups and Combinatorial Applications}.
\newblock Wiley-Interscience. J. Wiley \& Sons, Inc., New York, 1979.

\bibitem{Margolis&Rhodes&Schilling:arXiv:2406.18477}
S.~Margolis, J.~Rhodes, and A.~Schilling.
\newblock Decidability of {Krohn}-{Rhodes} complexity for all finite semigroups
  and automata.
\newblock Preprint, {arXiv}:2406.18477 [math.{GR}] (2024), 2024.

\bibitem{Margolis&Rhodes&Schilling:arXiv:2501.00770v1}
S.~Margolis, J.~Rhodes, and A.~Schilling.
\newblock {Complexity of Finite Semigroups: History and Decidability}.
\newblock Preprint, {arXiv}:2501.00770v1 [math.{GR}] (2025), 2025.

\bibitem{McKnight&Storey:1969}
J.~D. McKnight, Jr. and A.~J. Storey.
\newblock Equidivisible semigroups.
\newblock {\em J. Algebra}, 12:24--48, 1969.

\bibitem{Pin:1986;bk}
J.-E. Pin.
\newblock {\em Varieties of Formal Languages}.
\newblock Plenum, London, 1986.
\newblock English translation.

\bibitem{Reiterman:1982}
J.~Reiterman.
\newblock The {B}irkhoff theorem for finite algebras.
\newblock {\em Algebra Universalis}, 14:1--10, 1982.

\bibitem{Rhodes&Steinberg:2001}
J.~Rhodes and B.~Steinberg.
\newblock Profinite semigroups, varieties, expansions and the structure of
  relatively free profinite semigroups.
\newblock {\em Int. J. Algebra Comput.}, 11:627--672, 2001.

\bibitem{Rhodes&Steinberg:2006}
J.~Rhodes and B.~Steinberg.
\newblock Complexity pseudovarieties are not local; {T}ype {II} subsemigroups
  can fall arbitrarily in complexity.
\newblock {\em Int. J. Algebra Comput.}, 16(4):739--748, 2006.

\bibitem{Rhodes&Steinberg:2009qt}
J.~Rhodes and B.~Steinberg.
\newblock {\em The $q$-theory of finite semigroups}.
\newblock Springer Monographs in Mathematics. Springer, 2009.

\bibitem{Straubing:1979a}
H.~Straubing.
\newblock Aperiodic homomorphisms and the concatenation product of recognizable
  sets.
\newblock {\em J. Pure Appl. Algebra}, 15:319--327, 1979.

\bibitem{Tilson:1987}
B.~Tilson.
\newblock Categories as algebra: an essential ingredient in the theory of
  monoids.
\newblock {\em J. Pure Appl. Algebra}, 48:83--198, 1987.

\bibitem{Whyburn:1964}
G.~T. Whyburn.
\newblock {\em Topological Analysis}.
\newblock Princeton University Press, 1964.

\bibitem{Willard:1970}
S.~Willard.
\newblock {\em General Topology}.
\newblock Addison-Wesley, Reading, Mass., 1970.

\end{thebibliography}
\end{document}